\documentclass{article}

\usepackage{amsmath}
\usepackage{amsfonts}
\usepackage{amssymb}
\usepackage{amsthm}
\usepackage{mathtools}
\usepackage{graphicx}
\usepackage{color}
\usepackage{fancyhdr}
\usepackage{geometry}
\usepackage{hyperref}
\usepackage{mathrsfs}
\usepackage[nottoc,numbib]{tocbibind}

\usepackage{verbatim}

\allowdisplaybreaks

\makeatletter
\DeclareRobustCommand{\format@sec@number}[2]{{\normalfont\upshape#1}#2}

\makeatother

\def\e{\varepsilon}

\def\a{\alpha}
\def\b{\beta}

\def\d{\delta}

\def\l{\lambda}
\def\L{\Lambda}

\def\s{\sigma}

\def\R{\mathbb R}
\def\N{{\mathbb N}}

\def\Z{\mathbb Z}
\def\T{\mathbb T}

\def\K{\mathbb K}

\def\S{\mathbb S}

\def\C{\mathbb C}

\def\TT{\mathcal T}

\def\({\biggl(}
\def\){\biggr)}
\def\<{\mathbf{\langle}}
\def\>{\mathbf{\rangle}}

\def\diffr{\text{Diff }^\omega_\rho(\T^2,\mu)}

\def\Trees{\mathcal{T}rees}

\newcommand{\Meng}[2]{\left\{#1\mathrel{}\middle|\mathrel{}#2\right\}}
\newcommand{\abs}[1]{\left\lvert#1\right\rvert}

\numberwithin{equation}{section}

\newtheorem{theorem}[equation]{Theorem}
\newtheorem{proposition}[equation]{Proposition}
\newtheorem{lemma}[equation]{Lemma}
\newtheorem{corollary}[equation]{Corollary}
\newtheorem{fact}[equation]{Fact}

\newtheorem{maintheorem}{Theorem}

\theoremstyle{definition}
\newtheorem{definition}[equation]{Definition}

\theoremstyle{definition}

\theoremstyle{remark}
\newtheorem*{remark}{Remark}

\title{Real-analytic realization of Uniform Circular Systems and some applications}
\author{Shilpak Banerjee \footnote{ Indraprastha Institute of Information Technology, Delhi, India. Portion of this paper were completed when S.B. was partially supported by NSF grant DMS-16-02409 at Pennsylvania State University} \hspace{1cm} and \hspace{1cm} Philipp Kunde \footnote{Pennsylvania State University, Department of Mathematics, State College, PA 16802, USA. P.K. acknowledges financial support from a DFG Forschungsstipendium under Grant No. 405305501.}}

\graphicspath{{pic/}}

\begin{document}




\maketitle

\begin{abstract}
Recently Matthew Foreman and Benjamin Weiss showed in a series of papers that smooth ergodic diffeomorphisms of a compact manifold are unclassifiable up to measure-isomorphism. In this paper we show that the uniform circular systems used in the work of Foreman-Weiss admit real-analytic realizations on the torus. As a consequence we obtain the same anti-classification result for real-analytic ergodic diffeomorphisms on the torus. In another application we show the existence of an uncountable family of pairwise non-Kakutani equivalent real-analytic diffeomorphisms on the torus.
\end{abstract}

\tableofcontents

\section{Introduction}\label{section introduction}
In the foundational paper \cite{V} John von Neumann formulated the so-called \textit{isomorphism problem} for classifying the ergodic measure preserving transformations up to measure theoretic isomorphisms, i.e. it asks to determine when two measure-preserving transformations are isomorphic. It has been a guiding light for directions of research within ergodic theory, and has been solved only for some special classes of transformations, e.g.
\begin{itemize}
\item In 1942 von Neumann and Paul R. Halmos showed that the spectrum of the associated Koopman operator is a complete isomorphism invariant for ergodic measure preserving transformations with pure point spectrum (see \cite{HV}). 
\item In 1970 Donald Ornstein showed that Bernoulli shifts are completely classified by their entropy (see \cite{O}).
\end{itemize}
Many properties of transformations like mixing of various types or finite rank have been characterized and studied in connection with this problem but the general problem remained intractable. Starting in the late 1990's, so-called \textit{anti-classification} results have been established and demonstrated in a rigorous way that classification is not possible. This type of results requires a precise definition of what it means to obtain a  classification result. 

Informally, a classification is a method of determining isomorphism between transformations, perhaps by computing other invariants for which equivalence is easy to determine. Given a transformation (abstract or smooth or real-analytic) $T$ defined on a measure space (or appropriate manifold), one can ask whether or not it is possible to accurately describe the set $U$ of all other transformations (abstract or smooth or real-analytic) that are measure theoretically isomorphic to $T$. If a classification for such systems is available, then one can expect to proceed by computing invariants and reducing a set containing $T$ to smaller and smaller sets by requiring all elements to have a uniform value of the invariant  being considered. For example, one could start with a $T$ and look at the set containing all $S$ with the same entropy as that of $T$ and then in the next step, one could further restrict by requiring their Koopman operator to have the same set of eigenvalues and so on and so forth. So if one wants to prove that certain transformations are not classifiable, one would need to show that no countable (possibly transfinite) protocol, whose basic input is membership in open or closed sets exists that can be used to determine membership in a particular equivalence class. So in other words, the equivalence relation should not be Borel. So, the Borel/non-Borel distinction is a natural and key notion in anti-classification results:

\begin{itemize}
\item In 1996 Ferenc Beleznay and Matthew Foreman showed that the class of measure distal transformations used in early ergodic theoretic proofs of Szemeredi's theorem is not a Borel set \cite{BF}. 
\item In 2001 Greg Hjorth introduced the notion of \emph{turbulence} and showed that there is no Borel way of attaching algebraic invariants to ergodic transformations that completely determine isomorphism \cite{H}.
\item Foreman and Benjamin Weiss improved Hjorth's result by proving that the conjugacy action of the measure preserving transformations is turbulent and hence no generic class can have a complete set of algebraic invariants \cite{FW04}.
\item Passing from a single transformation to pairs $(S,T)$ of measure preserving Hjorth \cite{H} showed that the equivalence relation defined by isomorphism is not a Borel set. Since his proof uses nonergodic transformations in an essential way, this left open the question what happens if one restricts to ergodic transformations. This was resolved by Foreman, Daniel J. Rudolph and Weiss who showed that the measure-isomorphism relation for ergodic measure preserving transformations is also not a Borel set \cite{FRW}.
\end{itemize}

Recently, in a series of papers (\cite{FW1}, \cite{FW2} and \cite{FW3}) Foreman and Weiss showed an anti-classification result for $C^{\infty}$ diffeomorphisms of compact manifolds by proving that the measure-isomorphism relation among pairs of volume preserving ergodic $C^{\infty}$-diffeomorphisms is not a Borel set with respect to the $C^{\infty}$-topology. Actually, it is a complete analytic set (see Definitions \ref{def:complete} and \ref{def:analytic}). 

The goal of this article is to upgrade the aforementioned anti-classification result to the real-analytic category on $\T^2$. At this juncture, we point out that the set of all volume preserving real-analytic diffeomorphisms on a torus is not a metrizable space. In fact it is very difficult to work in this space and we prefer to do all construction on a certain subset $\text{Diff }^\omega_\rho(\T^2,\mu)$. For some pre-specified finite number $\rho>0$ this subset consists of all volume preserving real-analytic diffeomorphisms homotopic to the identity whose lift to $\R^2$ allows a complexification extending to a band of width $\rho$ in the imaginary direction (see Section \ref{section real-analytic}). This set has the structure of a Polish space and we can state the following theorem.

\begin{maintheorem} \label{theo:intro}
For every $\rho >0$ the measure-isomorphism relation among pairs $(S,T)\in\text{Diff }^\omega_\rho(\T^2,\mu)\times \text{Diff }^\omega_\rho(\T^2,\mu)$  is a complete analytic set and hence not a Borel set.
\end{maintheorem}

Note that the space of all measure preserving real-analytic diffeomorphisms, $\text{Diff }^\omega(\T^2,\mu)\coloneqq\cup_{\rho>0}\text{Diff }^\omega_\rho(\T^2,\mu)$ can be equipped with the direct limit topology, i.e. a set is $U\subset\text{Diff }^\omega(\T^2,\mu)$ is open iff $U\cap\text{Diff }^\omega_\rho(\T^2,\mu)$ is open in $\text{Diff }^\omega_\rho(\T^2,\mu)$. See \cite[appendix 2]{L} for a detailed discussion of these spaces. However the space obtained in this way in not metrizable. We can see the following result as an immediate consequence of theorem \ref{theo:intro}.

\begin{maintheorem} \label{theo:intro2}
The measure-isomorphism relation among pairs $(S,T)\in\text{Diff }^\omega(\T^2,\mu)\times\text{Diff }^\omega(\T^2,\mu)$  is not a Borel set, when $\text{Diff }^\omega(\T^2,\mu)$ is equipped with the direct limit topology.
\end{maintheorem}

The Foreman-Weiss anti-classification result and Theorem \ref{theo:intro} are related to another major question in ergodic theory dating back to the pioneering paper \cite{V}: The \textit{smooth realization problem} asks whether there are smooth versions to the objects and concepts of abstract ergodic theory and whether every ergodic measure-preserving transformation has a smooth model. Here, by a smooth model one means a smooth diffeomorphism of a compact manifold preserving a measure equivalent to the volume element which is isomorphic to the measure-preserving transformation. The only known restriction is due to A. G. Kushnirenko who proved that such a diffeomorphism must have finite entropy. On the other hand, there is a lack on general results on the smooth realization problem.

One way to show that not all finite ergodic measure preserving transformation have a smooth model would be to show that their classification is easier than the general classification result. But the Foreman-Weiss anti-classification result shows that the variety of ergodic transformations that have smooth models is rich enough so that the abstract isomorphism relation restricted to these smooth models is as complicated as it is in general. Theorem \ref{theo:intro} shows that this even holds true for real-analytic diffeomorphisms of $\T^2$.

One of the key steps in the recent work by Foreman and Weiss to adapt the methods from \cite{FRW} to the case of volume preserving $C^{\infty}$ diffeomorphisms is to show that a class of symbolic systems, the so-called strongly uniform circular systems (see Section \ref{subsec:circular}), can be realized as $C^{\infty}$ diffeomorphisms. To obtain these smooth realizations the so-called untwisted \emph{AbC method} is used. This method was introduced by D.V. Anosov and A. Katok in their seminal paper \cite{AK} and it is also widely known as the \emph{approximation by conjugation} method or the \emph{Anosov-Katok} method. In this paper we prove the real-analytic counterpart of the main result in \cite{FW1} that can be \emph{very loosely} summarized as the following theorem.  

\begin{maintheorem} \label{theo:loosely realization}
Let $T$ be an ergodic transformation on a standard measure space. Then the following are equivalent:
\begin{enumerate}
\item $T$ is isomorphic to an real-analytic (untwisted) Anosov-Katok diffeomorphim (satisfying some requirements).
\item $T$ is isomorphic to a (strongly uniform) circular system (with fast growing parameters) .
\end{enumerate}
\end{maintheorem}

For an accurate version of the above theorem stated with all the required technicalities we refer the reader to Theorem \ref{theorem 57}. Since this general realization of uniform circular systems as diffeomorphisms reduces questions about diffeomorphisms to combinatorial questions for symbolic shifts, it bears a lot of flexibility to address questions in the realm of the smooth realization problem. 

To exemplify the flexibility and strength of the real-analytic realization of circular systems we construct an uncountable family of real-analytic ergodic diffeomorphisms that are pairwise non-Kakutani equivalent. Recall that two ergodic transformations are said to be Kakutani equivalent if they are isomorphic to measurable cross-sections of the same ergodic flow.  Then it immediately follows from Abramov’s entropy formula, that two Kakutani equivalent automorphisms must have the same entropy type: zero entropy, finite entropy, or infinite entropy. It was a long-standing open problem whether these three possibilities for entropy completely characterized Kakutani equivalence classes. Until the work of A. Katok \cite{K75,K77} in the case of zero entropy, and J. Feldman \cite{Fe} in the general case, no other restrictions were known for achieving Kakutani equivalence. In \cite{Fe}, Feldman showed that there are at least two non-Kakutani equivalent ergodic transformations of entropy zero,
and likewise for finite positive entropy and infinite entropy. Ornstein, Rudolph, and Weiss \cite{ORW} upgraded Feldman's construction to obtain uncountably many  non-Kakutani equivalent ergodic MPT's of each entropy type. Based on their construction, M. Benhenda \cite{Ben} showed the existence of an uncountable family of pairwise non-Kakutani equivalent zero-entropy $C^{\infty}$ diffeomorphisms using the AbC method. As already indicated in \cite[Chapter 8]{Ka03}, this result also follows from the smooth realization of uniform circular systems in general. We present the details in Section \ref{subsec:uncountable} and combine it with our real-analytic realization of uniform circular systems to obtain the analogue in the real-analytic category.

\begin{maintheorem} \label{theo:uncountable}
For every $\rho>0$ there exists an uncountable family of ergodic diffeomorphisms in $\text{Diff }^\omega_\rho(\T^2,\mu)$ which are pairwise not Kakutani-equivalent.
\end{maintheorem}

In Section \ref{subsec:nonBern} we use this result to construct real-analytic non-Bernoulli diffeomorphisms with property $K$ on manifolds of dimension greater than $4$. Note that a Bernoulli-automorphism (i.\,e. a measure-preserving invertible transformation of a Lebesgue space isomorphic to a Bernoulli shift) has the $K$-property (i.\,e. a measure-preserving invertible transformation of a Lebesgue space obeying Kolmogorov's zero-one law) automatically. Originally, Kolmogorov conjectured that the converse holds also true. The first counterexample in the measurable category was constructed by Ornstein \cite{Or}. First $C^{\infty}$ counterexamples were obtained by Katok in dimension greater than $4$ \cite{Ka1}. Recently, A.\ Kanigowski, F.\ Rodriguez-Hertz, and K.\ Vinhage have found $C^{\infty}$ examples in dimension $4$ \cite{KRV}. On the other hand, smooth K-automorphisms in dimension $2$ are Bernoulli by Pesin theory \cite{Pe1}.

The first real-analytic examples of non-Bernoulli diffeomorphisms with property $K$ were obtained by Rudolph in \cite{Ru} as follows: Take a hyperbolic toral automorphism $S$ on $\mathbb{T}^2$ given as $\{(\theta,\eta): 0\leq \theta,\eta <1\}$ and suppose $T_t$ is the geodesic flow on the tangent space $TM$ to a compact hyperbolic surface. Now define
\[
\hat{S}((\theta,\eta),y) = (S(\theta,\eta),T_{\sin(\theta)}(y)),
\]
which is a real analytic $K$-diffeomorphisms and not Bernoulli. 

Actually, Theorem \ref{theo:uncountable} enables us to show the existence of uncountably many real-analytic diffeomorphisms measure theoretically
K, with the same entropy, and pairwise not isomorphic (not even Kakutani equivalent).

\subsection{Statement of anti-classification results}
To state our anti-classification results precisely we will need some notions and concepts from \textit{Descriptive Set Theory}. The main tool is the idea of a \emph{reduction}.

\begin{definition}
Let $X$ and $Y$ be Polish spaces and $A \subseteq X$, $B \subseteq Y$. A function $f:X \to Y$ \emph{reduces} $A$ to $B$ if and only if for all $x \in X$: $x\in A$ if and only if $f(x) \in B$. \\
Such a function $f$ is called a Borel (resp. continuous) reduction if $f$ is a Borel (resp. continuous) function.
\end{definition}

$A$ being reducible to $B$ can be interpreted as saying that $B$ is at least as complicated as $A$. We note that if $B$ is Borel and $f$ is a Borel reduction, then $A$ is also Borel. Taking the converse of this statement we get that if $A$ is not Borel, then $B$ is not Borel.

\begin{definition}\label{def:complete}
If $\mathcal{S}$ is a collection of sets and $B \in \mathcal{S}$, then $B$ is called \emph{complete} for Borel (resp. continuous) reductions if and only if every $A \in \mathcal{S}$ is Borel (resp. continuously) reducible to $B$.
\end{definition}
Following the interpretation above $B$ is said to be at least as complicated as each set in $\mathcal{S}$. 

\begin{definition}\label{def:analytic}
If $X$ is a Polish space and $B \subseteq X$, then $B$ is \emph{analytic} if and only if it is the continuous image of a Borel subset of a Polish space. Equivalently, there is a Polish space $Y$ and a Borel set $C\subseteq X \times Y$ such that $B$ is the $X$-projection of $C$.
\end{definition}

There are analytic sets that are not Borel. We use a canonical example of such a set: the collection of ill-founded trees. 

To introduce those, we consider the set $\mathbb{N}^{<\mathbb{N}}$
of finite sequences of natural numbers. A \emph{tree }is a set $\mathcal{T}\subseteq\mathbb{N}^{<\mathbb{N}}$
such that if $\tau=\left(\tau_{1},\dots,\tau_{n}\right)\in\mathcal{T}$
and $\sigma=\left(\tau_{1},\dots,\tau_{s}\right)$ with $s\leq n$
is an initial segment of $\tau$, then $\sigma\in\mathcal{T}$. If
$\sigma$ is an initial segment of $\tau$, then $\sigma$ is a \emph{predecessor}
of $\tau$ and $\tau$ is a \emph{successor} of $\sigma$. We define
the level $s$ of a tree $\mathcal{T}$ to be the collection of elements
of $\mathcal{T}$ that have length $s$. 

An \emph{infinite branch} through $\mathcal{T}$ is a function $f:\mathbb{N}\to\mathbb{N}$
such that for all $n\in\mathbb{N}$ we have $\left(f(0),\dots,f(n-1)\right)\in\mathcal{T}$.
If a tree has an infinite branch, it is called \emph{ill-founded}.
If it does not have an infinite branch, it is called \emph{well-founded}.

In the following, let $\left\{ \sigma_{n}:n\in\mathbb{N}\right\} $
be an enumeration of $\mathbb{N}^{<\mathbb{N}}$ with the property
that every proper predecessor of $\sigma_{n}$ is some $\sigma_{m}$
for $m<n$. Under this enumeration subsets $S\subseteq\mathbb{N}^{<\mathbb{N}}$
can be identified with characteristic functions $\chi_{S}:\mathbb{N}\to\left\{ 0,1\right\} $.
The collection of such $\chi_{S}$ can be viewed as the members of
on infinite product space $\left\{ 0,1\right\} ^{\mathbb{N}^{<\mathbb{N}}}$
homeomorphic to the Cantor space. Here, each function $a:\left\{ \sigma_{m}:m<n\right\} \to\left\{ 0,1\right\} $
determines a basic open set 
\[
\left\langle a\right\rangle =\left\{ \chi:\chi\upharpoonright\left\{ \sigma_{m}:m<n\right\} =a\right\} \subseteq\left\{ 0,1\right\} ^{\mathbb{N}^{<\mathbb{N}}}
\]
and the collection of all such $\left\langle a\right\rangle $ forms
a basis for the topology. The collection of trees is a closed (hence
compact) subset of $\left\{ 0,1\right\} ^{\mathbb{N}^{<\mathbb{N}}}$
in this topology. Moreover, the collection of trees containing arbitrarily
long finite sequences is a dense $\mathcal{G}_{\delta}$ subset. In
particular, this collection is a Polish space. We will denote the
space of trees containing arbitrarily long finite sequences by $\mathcal{T}rees$.

Since the topology on the space of trees was introduced via basic
open sets giving us finite amount of information about the trees in
it, we can characterize continuous maps defined on $\mathcal{T}rees$
as follows.
\begin{fact}
\label{fact:contTree}Let $Y$ be a topological space. Then a map
$f:\mathcal{T}rees\to Y$ is continuous if and only if for all open
sets $O\subseteq Y$ and all $\mathcal{T}\in\mathcal{T}rees$ with
$f(\mathcal{T})\in O$ there is $M\in\mathbb{N}$ such that for all
$\mathcal{T}^{\prime}\in\mathcal{T}rees$ we have: if $\mathcal{T}\cap\left\{ \sigma_{n}:n\leq M\right\} =\mathcal{T}^{\prime}\cap\left\{ \sigma_{n}:n\leq M\right\} $,
then $f\left(\mathcal{T}^{\prime}\right)\in O$. 
\end{fact}

As advertised the following classical fact (see e.g. \cite{Ke}) gives us an example
of a complete analytic set.
\begin{fact}\label{fact:ill} Let $\mathcal{T}rees$ be the space of trees.
\begin{enumerate}
    \item The collection of ill-founded trees is a complete
analytic subset of $\mathcal{T}rees$.
\item The collection of trees that have at least two distinct infinite branches is a complete
analytic subset of $\mathcal{T}rees$.
\end{enumerate}
\end{fact}

We also recall that the \emph{centralizer} of an invertible measure preserving transformation $T:(X,\mu)\to (X,\mu)$ is defined as $C(T)=\Meng{S:X\to X \text{ invertible m.p. transformation}}{S\circ T = T\circ S}$. Obviously, the powers $T^k$ belong to $C(T)$ for any $k\in \Z$. We also stress that we consider the centralizer in the group $\mathbf{MPT}$ of invertible measure preserving transformations and that $C(T)$ differs from the centralizer inside the group of diffeomorphisms.

Now we can present the precise statement of the main result of the paper that we will prove in Section \ref{section main proof}.

\begin{maintheorem} \label{theo:main anti}
For every $\rho >0$ there is a continuous function $F^s:\Trees \to \diffr$ such that for $\TT \in \Trees$, if $T=F^s(\TT)$:
\begin{enumerate}
    \item $\TT$ has an infinite branch if and only if $T \cong T^{-1}$
    \item $\TT$ has two distinct infinite branches if and only if 
    \[
    C(T) \neq \overline{ \Meng{T^n}{n\in \Z}}.
    \]
\end{enumerate}
\end{maintheorem}

Then the classical Fact \ref{fact:ill} implies the following anti-classification results.

\begin{maintheorem}
For every $\rho >0$ we have:
\begin{enumerate}
    \item $\Meng{T\in\diffr}{T\cong T^{-1}}$ is complete analytic.
    \item $\Meng{T\in\diffr}{C(T) \neq \overline{ \Meng{T^n}{n\in \Z}}}$ is complete analytic.
\end{enumerate}
\end{maintheorem}

Since the map $i(T)=(T,T^{-1})$ is a continuous mapping from $\diffr$ to $\diffr \times \diffr$ and reduces $\Meng{T}{T\cong T^{-1}}$ to $\Meng{(S,T)}{S\cong T}$, this result also yields that the measure-isomorphism relation among pairs $(S,T)$ of ergodic $\diffr$-diffeomorphisms of $\T^2$ is a complete analytic set and, hence, we obtain Theorem \ref{theo:intro} as a consequence.

\subsection{Outline of the paper}
The theme of this paper is to make certain changes to the Foreman-Weiss' series of paper in order to obtain the same anti-classification result for real-analytic diffeomorphisms of $\T^2$. In this connection, we also intend to give a survey of their impressive work.

In their proof of the anti-classification result for ergodic measure-preserving transformations in \cite{FRW} Foreman, Rudolph and Weiss construct a continuous function from the space $\Trees$ to the invertible measure-preserving transformations assigning to each tree $\mathcal{T}$ a transformation $T=F(\mathcal{T})$ of finite entropy such that $T \cong T^{-1}$ just in case $\mathcal{T}$ has an infinite branch (see Section \ref{subsec:FRWmap}). These transformations have an odometer as a factor and, hence, are \emph{Odometer-based Systems} in up-to-date terminology (see Section \ref{subsec:odom}). Since it is a persistent open problem to find a smooth realization of transformations with an odometer-factor, Foreman and Weiss circumvent that obstacle by showing that the collection of Odometer-based Systems has the same global structure with respect to joinings as another collection of transformations, the so-called \emph{Circular Systems} that are extensions of particular circle rotations (see Section \ref{subsec:circular} for a precise description). For this purpose, they show in \cite{FW2} that there is a functor $\mathcal{F}$ between these classes that takes synchronous joinings to synchronous joinings, anti-synchronous joinings to anti-synchronous joinings as well as synchronous and anti-synchronous isomorphisms to synchronous and anti-synchronous isomorphisms (see Section \ref{subsec:catfunc} for an explanation of the terms synchronous and anti-synchronous).

The definition of these circular systems is inspired by a symbolic representation for circle rotations by certain Liouville rotation numbers found in \cite{FW1}. Then Foreman and Weiss showed that these circular systems can be realized as volume preserving ergodic $C^{\infty}$-diffeomorphisms on a torus or a disk or an annulus (under some assumptions on the circular coefficients).

To obtain these smooth realizations the so-called untwisted \emph{AbC method} is used. In their seminal paper \cite{AK} D.V. Anosov and A. Katok introduced this method for constructing examples of $C^{\infty}$-diffeomorphisms. The construction can be carried out on any smooth compact connected manifold that admits an effective circle action $\mathcal{R}=\{R_t\}_{t\in \S^1}$. It involves inductively constructing measure preserving diffeomorphisms $T_n=H_n\circ  R^{\a_n}\circ H_n^{-1}$ where the diffeomorphism $H_n$ is obtained as the compositions $H_n=H_{n-1}\circ h_n$ and the rational number $\a_n$ is chosen close enough to $\a_{n-1}$ to guarantee convergence of the sequence $T_n$ to a diffeomorphism $T$. Usually we want $T$ to satisfy some dynamical property like weak-mixing, minimality, unique ergodicity, etc. and this is achieved by constructing $h_n$ at the $n$-th step in a way so that $T_n$ satisfies some finite version of the targeted property.  

Such constructions have resulted in several interesting results in the $C^{\infty}$-world. Beyond smoothness, the next natural question is the setting of real-analytic realizations. However, there are fewer success stories for real analytic diffeomorphisms due to certain well documented difficulties (see e.g. \cite[section 7.1]{FK} as well as \cite[Section 6.3]{Ba-Ku}). In recent years new papers have appeared which demonstrate that such constructions are possible with enough flexibility on the torus (\cite{S}, \cite{Ba-Ns}, \cite{Kana}, \cite{Ba-Ku}). A modified version of the construction also shows potential for real-analytic constructions beyond the torus (\cite{FK-ue}).

In the case of tori $\T^d$, $d \geq 2$, the concept of \emph{block-slide type of maps} introduced in \cite{Ba-Ns} and their sufficiently precise approximation by volume preserving real-analytic diffeomorphisms allows to find real-analytic counterparts of several Anosov-Katok constructions (see \cite{Ba-Ku}). This approach is the important mechanism in the constructions of this paper as well (we emphasize that all constructions in this article are done on the torus and that real-analytic AbC constructions on arbitrary real-analytic manifolds continue to remain an intractable problem). Hereby, we get our real-analytic counterpart of the main result in \cite{FW1} in Theorem \ref{theorem 57} that we already advertised in a vague manner as Theorem \ref{theo:loosely realization}. Also, unlike Foreman-Weiss, we use $\T^2$ as our manifold and also as our abstract measure space. This is done to reduce notational complexity.

\section{Preliminaries}

Here we introduce the basic concepts and establish notations that we will use for the rest of this article. We note that section \ref{section mpt} and \ref{section pp} are standard theory presented from \cite[sections 3.1 and 5]{FW1} and hence we skip all proofs. Section \ref{section real-analytic} is presented with complete proofs since the theory is somewhat rare. However one can find similar or identical exposition in \cite{S},\cite{FS},\cite{Ba-Ns} and \cite{Ba-Ku}.

\subsection{Basics in Ergodic Theory} \label{section mpt}


We give an concise survey of some concepts regarding measure spaces and measure preserving transformations. Our objective here is not to be comprehensive but rather to introduce some well known concepts in the context of our article and establish certain notations.

Let $\mathbb{X}$ be a set, $\mathcal{B}$ a Boolean algebra of measurable sets on $\mathbb{X}$ and $m$ a measure. Then if the triplet $(\mathbb{X},\mathcal{B},m)$ is a separable non-atomic probability space, we call it a \emph{standard measure space}. Let $(\mathbb{X},\mathcal{B},m)$ and $(\mathbb{X}',\mathcal{B}',m')$ be two measure space. Then a map $f$ defined on a set of full $m$ measure of $\mathbb{X}$ onto a set of full $m'$ measure of $\mathbb{X}'$ is called a \emph{measure theoretic isomorphism} if $f$ is one-one and both $f$ and $f^{-1}$ are measurable.

An \emph{ordered countable partition} or simply a  \emph{partition}
of $(\mathbb{X},\mathcal{B},m)$ will refer to a sequence $\mathcal{P}:=\{P_n\}_{n=1}^\infty$ such that the following conditions are satisfied:
\begin{enumerate}
\item $P_n\in\mathcal{B}$ for all $n\in\N$.
\item $P_n\cap P_m=\emptyset$ if $n\neq m$.
\item $m(\cup_{n=1}^\infty P_n)=1$.
\end{enumerate}
Each $P_n$ is called an \emph{atom} of the partition $\mathcal{P}$. Often we deal with partitions where all but finitely many atoms have measure zero and we refer to such a partition as a \emph{finite partition}. We say that a sequence of partitions $\{\mathcal{P}_n\}_{n=1}^\infty$ is a \emph{generating} sequence if the smallest $\sigma$-algebra containing $\cup_{n=1}^\infty\mathcal{P}_n$ is $\mathcal{B}$.

Next we introduce the notion of a distance between two partitions. Let $\mathcal{P}=\{P_n\}_{n=1}^\infty$ and $\mathcal{Q}=\{Q_n\}_{n=1}^\infty$ be two partitions, we define 
\begin{align}
D_m(\mathcal{P},\mathcal{Q}):=\sum_{i=1}^\infty m(P_i\triangle Q_i)
\end{align}

\begin{lemma} \label{lemma isomorphism}
Fix a sequence $\{\e_n\}_{n=1}^\infty$ such that $\sum_{n=1}^\infty\e_n<\infty$. Let $(\mathbb{X},\mathcal{B},m)$ and $(\mathbb{X}',\mathcal{B}',m')$ be two standard measure spaces and $\{T_n\}_{n=1}^\infty$ and $\{T_n'\}_{n=1}^\infty$ be measure preserving transformations of $\mathbb{X}$ and $\mathbb{X}'$ converging weakly \footnote{ The set of all measure preserving transformation on $\mathbb{X}$ has a natural topology where a basic open set is given by the sets $N(T,\mathcal{P},\e)\coloneqq\{S:$ $S$ is a measure preserving transformation and $\sum_{A\in\mathcal{P}}m(TA\triangle SA)<\e\}$ for a measure preserving transformation $T$, a finite measurable partition $\mathcal{P}$ and some $\e>0$. This is called the weak topology.} to $T$ and $T'$ respectively. Suppose $\{\mathcal{P}_n\}_{n=1}^\infty$ is a decreasing sequence of partitions and $\{K_n\}_{n=1}^\infty$ is s sequence of measure preserving transformations such that 
\begin{enumerate}
\item $K_n:\mathbb{X}\to\mathbb{X}'$ is an isomorphism between $T_n$ and $T_n'$.
\item $\{\mathcal{P}_n\}_{n=1}^\infty$ and $\{K_n(\mathcal{P}_n)\}_{n=1}^\infty$ are generating sequence of partitions for $\mathbb{X}$ and $\mathbb{X}'$.
\item $D_m(K_{n+1}(\mathcal{P}_n),K_n(\mathcal{P}_n))<\e_n$. 
\end{enumerate}  
Then the sequence $K_n$ converges in the weak topology to a measure theoretic isomorphism between $T$ and $T'$.
\end{lemma}

\subsection{Periodic processes} \label{section pp}

Here we recall some relevant facts from the notion of periodic processes. One can refer to \cite{Ka03} for a detailed account. We continue with notations from the previous subsection. 

Let $\mathcal{P}$ be a partition of $(\mathbb{X},\mathcal{B},m)$ where all the atoms of $\mathcal{P}$ have the same measure. A \emph{Periodic process} is a pair $(\tau,\mathcal{P})$ where $\tau$ is a permutation of $\mathcal{P}$ such that each cycle has equal length \footnote{ This is not a necessary requirement for the definition, but is good enough for us.}. We refer to these cycles as \emph{towers} and their length is called height of the tower. We also choose an atom from each tower arbitrarily and call it the base of the tower. In particular if $\mathcal{P}_1, \ldots,\mathcal{P}_s$ are the towers (of height $q$) of this periodic process with $B_1,\ldots,B_s$ as their respective bases, then any tower $\mathcal{P}_i$ can be explicitly written as $B_i,\tau(B_i),\ldots,\tau^{q-1}(B_i)$. We refer to $\tau^k(B_i)$ as the $k$-th level of the tower $\mathcal{P}_i$. $\tau^{q-1}(B_i)$ is the top level. Next we describe how to compare two periodic processes.

\begin{definition}[$\e$-approximation]
Let $(\tau,\mathcal{P})$ and $(\sigma,\mathcal{Q})$ be two periodic processes of the measure space $(\mathbb{X},\mathcal{B},m)$. We say that $(\sigma,\mathcal{Q})$ $\e$-approximates $(\tau,\mathcal{P})$ if there exists disjoint collections of $\mathcal{Q}$ atoms $\{S_A: A\in\mathcal{P}, S_A\subset \mathcal{Q}\}$. and a set $D\subset\mathbb{X}$ of measure less than $\e$ such that the following are satisfied:
\begin{enumerate}
\item For every $A\in\mathcal{P}$, we have $\cup\{B:B\in S_A\}\setminus D\subset A$.
\item If $A\in\mathcal{P}$ is not the top level of a tower and $B\in S_A$, we have $\sigma(B)\setminus D\subseteq \tau(A)$
\item For each tower of $\sigma$, the measure of the intersection of $\mathbb{X}\setminus D$ with each level of this tower are the same. 
\end{enumerate} 
\end{definition} 

With the above definition in mind, we have the following result regarding the convergence of periodic processes.


\begin{lemma}
Let $\{\e_n\}_{n\in \N}$ be a summable sequence of positive numbers. Let $\{(\tau_n,\mathcal{P}_n)\}_{n\in \N}$ be a sequence of periodic processes on $(\mathbb{X},\mathcal{B},m)$ such that $(\tau_{n+1},\mathcal{P}_{n+1})$ $\e_n$-approximates $(\tau_{n},\mathcal{P}_{n})$ and the sequence $\{\mathcal{P}_n\}_{n\in\N}$ is generating. 

Then there exists a unique transformation $T:\mathbb{X}\to\mathbb{X}$ satisfying:
\begin{align}
\lim_{n\to\infty} m(\bigcup_{A\in\mathcal{P}_n}(\tau_nA\triangle TA))=0
\end{align}
We call $\{(\tau_n,\mathcal{P}_n)\}$ a \emph{convergent sequence of periodic processes}.
\end{lemma}

\begin{lemma}
Let $\{(\tau_n,\mathcal{P}_n)\}_{n\in\N}$ be a sequence of periodic processes converging to $T$ and $\{T_n\}_{n\in\N}$ be a sequence of measure preserving transformations satisfying for each $n$, 
\begin{align}
\sum_{A\in\mathcal{P}_n}\mu(T_n A\triangle\tau_n A)<\e_n 
\end{align} 
Then $\{T_n\}_{n\in\N}$ converges weakly to $T$.
\end{lemma}

\begin{lemma}
Let $\{\e_n\}_{n\in \N}$ be a summable sequence of positive numbers. Let $(\mathbb{X}_1,\mathcal{B}_1, m_1)$ and $(\mathbb{X}_2,\mathcal{B}_2, m_2)$ be two standard measure spaces. Let $\{T_n\}_{n\in\N}$ and $\{S_n\}_{n\in\N}$ be two sequences of measure preserving transformations of $\mathbb{X}_1$ and $\mathbb{X}_2$ converging to $T$ and $S$ respectively in the weak topology. Suppose $\mathcal{P}_n$ is a decreasing sequence of partitions and $\{\phi_n\}_{n\in\N}$ be a sequence of measure preserving transformations such that
\begin{enumerate}
\item $\phi_n$ is a measure theoretic isomorphism between $T_n$ and $S_n$. 
\item The sequence $\mathcal{P}_n$ and $\phi_n(\mathcal{P}_n)$ generate $\mathcal{B}_1$ and $\mathcal{B}_2$.
\item $D_{m_2}(\phi_{n+1}(\mathcal{P}_n),\phi_n(\mathcal{P}_n))<\e_n$
\end{enumerate}
Then the sequence $\phi_n$ converges in the weak topology to an isomorphism between $T$ and $S$.
\end{lemma}

\subsection{Real-analytic diffeomorphisms on the torus} \label{section real-analytic}

We will denote the two dimensional torus by 
\begin{align}
\T^2:=\R^2/\Z^2
\end{align}
We use $\mu$ to denote the standard Lebesgue measure on $\T^2$.

We give a description of the space of measure preserving real-analytic diffeomorphisms on $\T^2$ that are homotopic to the identity. 

Any real-analytic diffeomorphism on $\T^2$ homotopic to the identity admits a lift to a map from $\R^2$ to $\R^2$ and has  the following form 
\begin{align} 
F(x_1, x_2)=(x_1+f_1(x_1, x_2), x_2+f_2(x_1,x_2))
\end{align}
where $f_i:\R^2\to \R$ are $\Z^2$-periodic real-analytic functions.  Any real-analytic $\Z^2$ periodic function on $\R^2$ can be extended as a holomorphic (complex analytic) function from some open complex neighborhood \footnote{We identify $\R^2$ inside $\C^2$ via the natural embedding $(x_1, x_2)\mapsto (x_1+i0, x_2+i0)$.} of $\R^2$ in $\C^2$. For a fixed $\rho>0$, we define the neighborhood
\begin{align}
\Omega_\rho:=\{(z_1,z_2)\in\C^2:|\text{Im}(z_1)|<\rho \text{ and  }|\text{Im}(z_2)|<\rho\}
\end{align}
and for a function $f$ defined on this set, put
\begin{align}
 \|f\|_\rho:=\sup_{(z_1,z_2)\in \Omega_\rho}|f((z_1,z_2))|
\end{align}
We define  $C^\omega_\rho(\T^2)$ to be the space of all $\Z^2$-periodic real-analytic functions on $\R^2$ that extends to a holomorphic function on $\Omega_\rho$ and $\|f\|_\rho<\infty$.

We define, $\text{Diff }^\omega_\rho(\T^2,\mu)$ to be the set of all measure preserving real-analytic diffeomorphisms of $\T^2$ homotopic to the identity, whose lift $F$ to $\R^2$ satisfies $f_i\in C^\omega_\rho(\T^2)$ and we also require that the lift $\tilde{F}(x)=(x_1+\tilde{f}_1(x), x_2+\tilde{f}_2(x))$ of its inverse to $\R^2$ to satisfy $\tilde{f}_i\in C^\omega_\rho(\T^2)$.\\
The metric in $\text{Diff }^\omega_\rho(\T^2,\mu)$ is defined by 
\begin{align*}
d_\rho(f,g)=\max\{\tilde{d}_\rho(f,g),\tilde{d}_\rho(f^{-1},g^{-1})\}\qquad\text{where}\qquad\tilde{d}_\rho(f,g)=\max_{i=1,2}\{\inf_{n\in\Z}\|f_i-g_i+n\|_\rho\}
\end{align*}
Next, with some abuse of notation, we define the following two spaces 
\begin{align}
C^\omega_\infty (\T^2)  :=& \cap_{n=1}^\infty C^\omega_n(\T^2) \label{6.789} \\
\text{Diff }^\omega_\infty (\T^2,\mu)  :=  &  \cap_{n=1}^\infty \text{Diff }^\omega_n(\T^2,\mu) \label{4.569}
\end{align}
We now list some properties of the above spaces that are going to be useful to us.
\begin{itemize}
\item Note that the functions in \ref{6.789} can be extended to $\C^2$ as entire functions. 
\item $\text{Diff }^\omega_\infty(\T^2,\mu)$ is closed under composition. To see this assume that $f,g\in \text{Diff }^\omega_\infty(\T^2,\mu)$ and let $F,G$ be their lifts to $\R^2$. Then note that $F\circ G$ is the lift of the composition $f\circ g$ (with $\pi:\R^2\to\T^2$ as the natural projection, $\pi\circ F\circ G=f\circ\pi\circ G=f\circ g\circ \pi$). Now for the complexification of $F$ and $G$ note that the composition $F\circ G(z)=(z_1+g_1(z)+f_1(G(z)), z_2+g_2(z)+f_2(G(z)))$. Since $g_i\in C^\omega_\infty (\T^2) $, we have for any $\rho,$ $\sup_{z\in\Omega_\rho}|\text{Im}(G(z))|\leq \max_i (\sup_{z\in\Omega_\rho}|\text{Im}(z_i)+\text{Im}(g_i(z))|) \leq \max_i (\sup_{z\in\Omega_\rho}|\text{Im}(z_i)|+\sup_{z\in\Omega_\rho}|\text{Im}(g_i(z))|)\leq \rho + \max_i (\sup_{z\in\Omega_\rho}|g_i(z)|)<\rho + const<\rho'<\infty$ for some $\rho'$. So, $\sup_{z\in\Omega_\rho}|g_i(z)+f_i(G(z))|\leq \sup_{z\in\Omega_\rho}|g_i(z)|+\sup_{z\in\Omega_\rho}|f_i(G(z))|<\infty$ since $g_i\in C^\omega_\infty (\T^2)$, $G(z)\in \Omega_{\rho'}$ and $f_i\in C^\omega_{\rho'}(\T^2)$. An identical treatment gives the result for the inverse. 
\item If $\{f_n\}_{n=1}^\infty\subset\text{Diff }^\omega_\rho(\T^2,\mu)$ is Cauchy in the $d_\rho$ metric, then $f_n$ converges to some $f\in \text{Diff }^\omega_\rho(\T^2,\mu)$. Indeed, uniform convergence guarantees analyticity and consideration of the inverses ensures that the result is a diffeomorphism. So this space is a Polish space  \footnote{ A Polish space is a separable completely metrizable topological space}.
\end{itemize}

Similarly we define the space of all real-analytic functions on $\T^2$ and all measure preserving real-analytic diffeomorphisms of $\T^2$ as follows:
\begin{align}
C^\omega(\T^2)  :=& \cup_{n=1}^\infty C^\omega_\frac{1}{n}(\T^2)  \\
\text{Diff }^\omega(\T^2,\mu)  :=  &  \cup_{n=1}^\infty \text{Diff }^\omega_\frac{1}{n}(\T^2,\mu) 
\end{align}

The space of all measure preserving real-analytic diffeomorphisms of $\T^2$ is given the corresponding inductive limit topology and it is not a metrizable space. This space is usually very difficult to work with and we will not use this for the rest of this article.

This completes the description of the analytic topology necessary for our construction. Also throughout this paper, the word ``diffeomorphism" will refer to a real-analytic diffeomorphism unless stated otherwise. Also, the word ``analytic topology" will refer to the topology of $\text{Diff }^\omega_\rho(\T^2,\mu)$ described above. See \cite{S} for a more extensive treatment of these spaces.

\section{Symbolic systems}

In this section we introduce the notion of a symbolic system in a way that is most convenient in representing AbC transformations. In particular we write about the uniform circular systems introduced by Foreman and Weiss. We skip proofs and recall the results only for one can refer to \cite[sections 3.3 and 4]{FW1} for the details. We also note that we try to stick to the notations used in \cite{FW1} but we do make some changes.

\subsection{Symbolic systems}

Let $\Sigma $ be a finite or countable alphabet endowed with the discrete topology. By $\Sigma^\Z$ we denote the space of bi-infinite sequences of alphabets from $\Sigma$ endowed with the product topology. The product topology makes this a totally disconnected separable space that is compact if $\Sigma$ is finite.

In order to get a better description of the of the product topology on this set we define for any $u=\<u_0,u_1,\ldots,u_{n-1}\>\in\Sigma^{<\infty}$,
\begin{align*}
C_k(u)\coloneqq \{f\in\Sigma^Z:f|_{[k,k+n)}=u\}
\end{align*}
Such sets are known as \emph{cylinder sets} and they generate the product topology of $\Sigma^\Z$. Next we define the (left) \emph{shift map}
\begin{align}
\text{sh}:\Sigma^\Z\to\Sigma^\Z\qquad\text{defined by}\qquad \text{sh}(f)(n)=f(n+1)
\end{align}
If $\mu$ is a shift invariant Borel measure then the system $(\Sigma^\Z,\mathcal{B},\mu,\text{sh})$ is called a \emph{symbolic system}. The closure of the support of $\mu $ is a shift invariant measure preserving system and we call it a \emph{symbolic shift} or a \emph{sub shift}. 

We recall that we can construct symbolic shifts from an arbitrary measure preserving transformation $(\mathbb{X},\mathcal{B},\mu,T)$. To see this, fix a partition $\mathcal{P}\coloneqq\{A_i:\; i\in I\}$ for some countable or finite $I$ and an alphabet $\Sigma\coloneqq\{a_i:i\in I\}$. We define $\phi:\mathbb{X}\to\Sigma^\Z$ by $\phi(x)(n) = a_i\iff T^nx\in A_i$. The $\phi^*\mu$ is an invariant measure and $(\Sigma^\Z,\mathcal{C},\phi^*\mu,\text{sh} )$ is a factor of $(\mathbb{X},\mathcal{B},\mu,T)$ with factor map $\phi$. If $\mathcal{P}$ is generating then $\phi$ is an isomorphism.

The above description of symbolic shift and coding is the most straight forward way, but to explicitly understand the resulting symbolic shift conjugate with an AbC transformation we need a step by step inductive procedure for describing symbolic shifts which in certain way emulates the AbC process. So we give an intrinsic definition of  the {symbolic shifts} using the notion of construction sequences. 

\begin{definition}
\label{def:ConstrSeq} A sequence of collection of words $\left(W_{n}\right)_{n\in\mathbb{N}}$,
where $\mathbb{N}=\left\{ 0,1,2,\dots\right\} $, satisfying the following
properties is called a \emph{construction sequence}:
\begin{enumerate}
\item for every $n\in\mathbb{N}$ all words in $W_{n}$ have the same length
$h_{n}$,
\item each $w\in W_{n}$ occurs at least once as a subword of each $w^{\prime}\in W_{n+1}$,
\item there is a summable sequence $\left(\varepsilon_{n}\right)_{n\in\mathbb{N}}$
of positive numbers such that for every $n\in\mathbb{N}$, every word
$w\in W_{n+1}$ can be uniquely parsed into segments $u_{0}w_{1}u_{1}w_{1}\dots w_{l}u_{l+1}$
such that each $w_{i}\in W_{n}$, each $u_{i}$ (called spacer or
boundary) is a word in $\Sigma$ of finite length and for this parsing
\[
\frac{\sum_{i=0}^{l+1}\lvert u_{i}\rvert}{h_{n+1}}<\varepsilon_{n+1}.
\]
\end{enumerate}
\end{definition}

Next we define the following subset of $\Sigma^\Z$:
\begin{align}
\mathbb{K}\coloneqq\{x\in\Sigma^\Z: \text{ if } & x= uwv \text{ for some } u,v\in\Sigma^\Z, w\in\Sigma^{<\infty} \text{ then } w\nonumber\\
&\text{ is a contiguous subword of some }w'\in W_n\text{ for some }n\}
\end{align}
We note that $\mathbb{K}$ is a closed shift invariant subset of $\Sigma^\Z$.

Before we talk about measures on $\mathbb{K}$, we need some technical definitions. We say that a construction sequence $W_n$ is \emph{uniform} if there exists a sequence of functions $\{d_n:W_n\to(0,1)\}_{n=1}^\infty$ such that for some summable sequence of positive numbers $\{\e_n\}_{n=1}^\infty$, we have for any choice of $w\in W_n$ and $w'\in W_{n+1}$,
\begin{align}
\Big|\frac{h_{n}}{h_{n+1}}(\#\{i:w=w_i\text{ where } w'=u_0w_0u_1w_1\ldots u_lw_lu_{l+1}\})-d_n(w)\Big|<\frac{1}{h_n}\e_{n+1}
\end{align}
We say that $\mathbb{K}$ is a \emph{uniform symbolic system} if it is build out of a uniform construction sequence. 

If it so happens that the number $\#\{i:w=w_i\text{ where } w'=u_0w_0u_1w_1\ldots u_lw_lu_{l+1}\}$ depends only on $n$ but not on $w$ or $w'$ then we refer to the construction sequence and $\mathbb{K}$ as \emph{strongly uniform}. Note that strong uniformity implies uniformity with $d_n (w) = \#\{i:w=w_i\text{ where } w'=u_0w_0u_1w_1\ldots u_lw_lu_{l+1}\}(h_n /h_{n+1} )$.

We say that a collection of finite words $\mathcal{W}$ is \emph{uniquely readable} iff whenever $u,v,w\in\mathcal{W}$ and $uv=pws$ then either $p$ or $s$ is the empty word.

Let $\mathbb{K}$ be a uniform symbolic system built out of a construction sequence $W_n$ where each $W_n$ is uniquely readable. We define
\begin{align}
S\coloneqq\{x\in\mathbb{K}:\;\exists\text{ natural number sequences }\{a_m\}_{m=1}^\infty,\; \{b_m\}_{m=1}^\infty\text{ satisfying } x|_{[a_m,b_m)}\in W_m\}
\end{align} 
and note that $S$ is a dense shift invariant $\mathcal{G}_\d$ set.

\begin{lemma}\label{lemma 11}
Let $\mathbb{K}$ be a uniform symbolic system built out of the  construction sequence $\{W_n\}_{n=1}^\infty$ in a finite alphabet $\Sigma$. Then the following holds:
\begin{enumerate}
\item $\mathbb{K}$ is the smallest shift invariant closed subset of $\Sigma^\Z$ such that 
\begin{align}
\mathbb{K}\cap C_0(w)\neq\emptyset\qquad\forall n\in\N\qquad w\in W_n
\end{align}
\item There exists a unique non-atomic shift invariant measure $\nu$ concentrated on S and this measure is ergodic.
\item For any $s\in S\subset \mathbb{K}$ and $w\in W_n$, the density of $\{k:w$ occurs in $s$ starting at $k\}$ exists and is equal to $\nu(C_0(w))$. Moreover $\nu(C_0(w))=d_n(w)/h_n$.
\end{enumerate}
\end{lemma}

\subsection{Odometer-based symbolic systems} \label{subsec:odom}

Let $\left(k_{n}\right)_{n\in\mathbb{N}}$ be a sequence of natural
numbers $k_{n}\geq2$ and 
\[
O=\prod_{n\in\mathbb{N}}\left(\mathbb{Z}/k_{n}\mathbb{Z}\right)
\]
be the $\left(k_{n}\right)_{n\in\mathbb{N}}$-adic integers. Then
$O$ has a compact abelian group structure and hence carries a Haar
measure $\lambda$. We define a transformation $T:O\to O$ to be addition
by $1$ in the $\left(k_{n}\right)_{n\in\mathbb{N}}$-adic integers
(i.e. the map that adds one in $\mathbb{Z}/k_{0}\mathbb{Z}$ and carries
right). Then $T$ is a $\lambda$-preserving invertible transformation
called \emph{odometer transformation} which is ergodic and has discrete
spectrum.

We now define the collection of symbolic systems that have odometer
systems as their timing mechanism to parse typical elements of the
system.
\begin{definition}
Let $\left(\mathtt{W}_{n}\right)_{n\in\mathbb{N}}$ be a uniquely
readable construction sequence with $\mathtt{W}_{0}=\Sigma$ and $\mathtt{W}_{n+1}\subseteq\left(\mathtt{W}_{n}\right)^{k_{n}}$
for every $n\in\mathbb{N}$. The associated symbolic shift will be
called an \emph{odometer-based system}.
\end{definition}

Thus, odometer-based systems are those built from construction sequences
$\left(\mathtt{W}_{n}\right)_{n\in\mathbb{N}}$ such that the words
in $\mathtt{W}_{n+1}$ are concatenations of a fixed number $k_{n}$
of words in $\mathtt{W}_{n}$. Hence, the words in $\mathtt{W}_{n}$
have length $\mathtt{h}_{n}$, where
\[
\mathtt{h}_{n}=\prod_{i=0}^{n-1}k_{i}
\]
if $n>0$, and $\mathtt{h}_{0}=1$. Moreover, the spacers in part
3 of Definition \ref{def:ConstrSeq} are all the empty words (i.e.
an odometer-based transformation can be built by a cut-and-stack construction
using no spacers).

\begin{remark}
According to anannouncement in \cite{FW2}, any finite entropy system that has an odometer factor can be represented as an odometer-based system
\end{remark}

\subsection{Circular symbolic systems} \label{subsec:circular}

We now describe a special class of symbolic shifts called circular symbolic systems. These systems were introduced by Foreman and Weiss in \cite[section 4]{FW1} and are specifically designed to serve as symbolic representations of untwisted AbC systems. We recall the results from their work but skip most proofs.

Consider natural numbers $k,l,p,q$ with $p$ and $q$ mutually prime. We note that for every $i=0,1,\ldots, q-1$, there exists a $0\leq j_i<q$ such that the following holds:
\begin{align}
pj_i=i\mod q
\end{align} 
we often use the notation 
\begin{align}
j_i = (p)^{-1}i\mod q
\end{align}
for brevity. Note that $q-j_i=j_{q-i}$.

In order to describe uniform circular systems, we start by defining an operator on alphabets. Let $\Sigma$ be any alphabet and $\{b,e\}$ be two letter not in $\Sigma$. Suppose $w_0,\ldots, w_{k-1}$ be a sequence of words constructed out of the alphabet $\Sigma\cup\{b,e\}$. We define the operator:
\begin{align} \label{formula C}
\mathcal{C}(w_0,w_1,\ldots, w_{k-1})\coloneqq \prod_{i=0}^{q-1}\prod_{j=0}^{k-1}(b^{j_{q-i}}w_j^{l-1}e^{j_i})
\end{align}
Note that
\begin{itemize}
\item  If $|w_i|=q $ for $i=0,\ldots, k-1$, then $|\mathcal{C}(w_0,\ldots,w_{k-1})|=klq^2$.
\item For every $e\in \mathcal{C}(w_0,\ldots,w_{k-1})$, there exists a $b$ to the left of it.
\item If for some $m>n$, there exists a $b$ at the $n$ th position followed by a $e$ at the $m$ th position and additionally we know that neither occurs inside a $w_i$ then there must exist a $w_i$ in between the $m$ th and the $n$ th position. 
\item The proportion of the word $w$ written in equation \ref{formula C} that belongs to the boundary is $1/l$. Moreover the proportion of the word that is within $q$ letters of the boundary is $3/l$. 
\end{itemize}

We also introduce some notions around this map. Suppose $w=\mathcal{C}(w_0,w_1,\ldots, w_{k-1})$. So $w$ consists of blocks  with $l-1$ copies of $w_i$ along with some $b$ s and $e$ s at the ends which are not inside $w_i$. We refer to the portion of $w$ in $w_i$ s to be the \emph{interior of $w$} and the $b$ s and $e$ s not in the $w_i$ s to be the \emph{boundary of $w$}. In a block of the form $w_i^{l-1}$, the first and the last occurrences of $w_i$ is called the \emph{boundary portion of the block $w_j^{l-1}$} and the other occurrences are called \emph{interior occurrences}.

\begin{lemma}\label{lemma u readable}
Let $\Sigma$ be a finite or countable alphabet and $u_0,\ldots,u_{k-1},v_0,\ldots, v_{k-1},w_0,\ldots, w_{k-1}$ are words of length $q<l/2$ constructed from the alphabet $\Sigma\cup\{b,e\}$. Put $u=\mathcal{C}(u_0,\ldots,u_{k-1}), v=\mathcal{C}(v_0,\ldots,v_{k-1})$ and $w=\mathcal{C}(w_0,\ldots,w_{k-1})$. If for some words $p$ and $s$ constructed from $\Sigma\cup\{b,e\}$ we have
\begin{align*}
uv=pws
\end{align*}
Then either $p=$ $empty$ $word$, $u=w$, $v=s$ or $s=$ $empty$ $word$, $u=p$, $v=w$.
\end{lemma}
Next we inductively define four sequences of natural numbers $\{k_n\}_{n=1}^\infty, \{l_n\}_{n=1}^\infty, \{p_n\}_{n=1}^\infty$  Define $k_0, l_0$ and $\{q_n\}_{n=1}^\infty$ as follows:
\begin{align}\label{para cons}
p_{n+1}=p_nq_nk_nl_n+1\qquad\qquad q_{n+1}=k_nl_nq_n^2
\end{align}
We note that the above relations makes $p_n$ and $q_n$ relatively prime. So we can define for $0\leq i<q_n$ a natural number $0\leq j_i<q_n $ such that 
\begin{align}
j_i=(p_n)^{-1}i\mod q_n
\end{align}

Now we are ready to build a construction sequence for our symbolic shift. We choose a nonempty finite or countable alphabet $\Sigma$ and choose two letters $b$ and $e$ not in $\Sigma$. We start or induction by putting $\mathcal{W}_0=\Sigma$. Now we assume that the induction has been carried out till the $n$ th step.

At the $n+1$ th step we we choose a set $P_{n+1}\subset (\mathcal{W}_n)^{k_n}$ and put 
\begin{align}\label{n+1 cons}
\mathcal{W}_{n+1}\coloneqq \{\mathcal{C}(w_0,\ldots,w_{k_n-1}):(w_0,\ldots,w_{k_n-1})\in P_{n+1}\}
\end{align}
It follows from lemma \ref{lemma u readable} that $\mathcal{W}_{n+1}$ is uniquely readable.

Next we introduce the concept of strong unique readability. We can view $\mathcal{W}_n$ as a collection of $\L_n$ letters and elements of $P_{n+1}$ can be viewed as  words constructed out of $\Lambda_n$. If $P_{n+1}$ is uniquely readable in the alphabet $\Lambda_n$, we say that the construction sequence satisfies the \emph{strong unique readability} assumption.

A construction sequence which satisfies \ref{n+1 cons}, uses parameters satisfying \ref{para cons}, and satisfies the strong unique readability assumption is called a \emph{circular construction sequence}.

\begin{lemma}
If the $\{l_n\}_{n=1}^\infty$ parameters of a circular construction sequence satisfies 
\begin{align}\label{3.4569}
\sum_{n=1}^\infty 1/l_n<\infty
\end{align}
and for each $n$ there exists a number $f_n$ such that each word $w\in\mathcal{W}_n$ occurs exactly $f_n$ times in each word in $P_{n+1}$ then the circular construction sequence is strongly uniform.
\end{lemma}

\begin{definition}\label{definition circular systems}
A symbolic shift $\mathbb{K}$ constructed from a circular construction system is called a \emph{circular system}. A symbolic shift $\mathbb{K}$ constructed from a (strongly) uniform circular construction system is called a \emph{(strongly) uniform circular system}. 
\end{definition}

\begin{lemma}
Let $\mathbb{K}$ be a circular system. Then 
\begin{enumerate}
\item A shift invariant measure $\nu$ concentrates on $S\subset \mathbb{K}$ iff $\nu$ concentrates on the collection of $s\in\mathbb{K}$ such that $\{i:s(i)\not\in \{b,e\}\}$ is unbounded in both $\Z^+$ and $\Z^-$ direction.
\item If $\mathbb{K}$ is a uniform circular system and $\nu$ is a shift invariant measure on $\mathbb{K}$, then $\nu(S)=1$. In particular, there is a unique non-atomic shift-invariant measure on $\mathbb{K}$ by Lemma 3.7.
\end{enumerate}
\end{lemma}

We end this section after defining a canonical factor of a circular system measure theoretically isomorphic to a rotation of the circle. Let $\{k_n\}_{n=1}^\infty$ and $\{l_n\}_{n=1}^\infty$ be two parameter sequences satisfying \ref{3.4569}. 
Let $\mathbb{K}$ be any circular system constructed out of these parameters. We also construct a second circular system. Let $\Sigma_0=\{\ast\}$ and we define a construction sequence 
\begin{align}
\mathcal{W}_0:=\Sigma_0,\qquad\text{and if }\qquad\mathcal{W}_n:=\{w_n\}\quad\text{then}\quad\mathcal{W}_{n+1}=\{\mathcal{C}(w_n,\ldots,w_n)\}.
\end{align}
We denote the resulting circular system by $\mathcal{K}$. 

Now we claim that $\mathcal{K}$ is a factor of $\mathbb{K}$. We see that one can construct an explicit factor map as follows
\begin{align}
\pi:\mathbb{K}\to\mathcal{K}\qquad\text{defined by}\qquad \pi(x)\coloneqq\begin{cases} x(i) & \quad\text{if}\quad x(i)\in\{b,e\}\\ \ast & \quad\text{otherwise}
\end{cases}
\end{align}
We end this section with the following observations:
\begin{itemize}
\item $\pi:\mathbb{K}\to\mathcal{K}$ is Lipschitz.
\item $\pi\circ\text{sh}^\pm=\text{sh}^\pm\circ\pi$.
\item $\pi$ is a factor map of $\mathcal{K}$ to $\mathbb{K}$ and from $\mathcal{K}^{-1}$ to $\mathbb{K}^{-1}$.
\end{itemize}

In \cite[Section 4.3]{FW2} a specific isomorphism $\natural:\mathcal{K} \to rev(\mathcal{K})$ is introduced. It is called the \emph{natural map} and will serve as a benchmark for understanding of maps from $\mathbb{K}$ to $rev(\mathbb{K})$ (see e.g. Definition \ref{def:synchronous}). 

\subsection{Categories $\mathcal{OB}$ and $\mathcal{CB}$ and the functor $\mathcal{F}:\mathcal{OB}\to \mathcal{CB}$} \label{subsec:catfunc}

For a fixed circular coefficient sequence $\left(k_{n},l_{n}\right)_{n\in\mathbb{N}}$
we consider two categories $\mathcal{OB}$ and $\mathcal{CB}$ whose
objects are odometer-based and circular systems respectively. The
morphisms in these categories are (synchronous and anti-synchronous)
graph joinings. 

\begin{definition}\label{def:synchronous} If $\mathbb{K}$ is an odometer based system, we denote its odometer base by $\mathbb{K}^{\pi}$ and let $\pi:\mathbb{K}\to \mathbb{K}^{\pi}$ be the canonical factor map. Similarly, if $\mathbb{K}^c$ is a circular system, we let $(\mathbb{K}^c)^{\pi}$ be the rotation factor $\mathcal{K}$ and let $\pi:\mathbb{K}^c\to \mathcal{K}$ be the canonical factor map.
\begin{enumerate}
    \item Let $\mathbb{K}$ and $\mathbb{L}$ be odometer based systems with the same coefficient sequence and let $\rho$ be a joining between $\mathbb{K}$ and $\mathbb{L}^{\pm 1}$. Then $\rho$ is called \emph{synchronous} if $\rho$ joins $\mathbb{K}$ and $\mathbb{L}$ and the projection of $\rho$ to a joining on $\mathbb{K}^{\pi}\times \mathbb{L}^{\pi}$ is the graph joining determined by the identity map. The joining $\rho$ is called \emph{anti-synchronous} if $\rho$ joins $\mathbb{K}$ and $\mathbb{L}^{-1}$ and the projection of $\rho$ to a joining on $\mathbb{K}^{\pi}\times (\mathbb{L}^{-1})^{\pi}$ is the graph joining determined by the map $x\mapsto -x$.
    \item Let $\mathbb{K}^c$ and $\mathbb{L}^c$ be circular systems with the same coefficient sequence and let $\rho$ be a joining between $\mathbb{K}^c$ and $(\mathbb{L}^c)^{\pm 1}$. Then $\rho$ is called \emph{synchronous} if $\rho$ joins $\mathbb{K}^c$ and $\mathbb{L}^c$ and the projection of $\rho$ to a joining on $\mathcal{K}\times \mathcal{L}$ is the graph joining determined by the identity map. The joining $\rho$ is called \emph{anti-synchronous} if $\rho$ joins $\mathbb{K}^c$ and $(\mathbb{L}^c)^{-1}$ and the projection of $\rho$ to a joining on $\mathcal{K}\times \mathcal{L}^{-1}$ is the graph joining determined by the map $rev(\cdot)\circ \natural$.
\end{enumerate}
\end{definition}
In \cite{FW2} Foreman and Weiss define a functor taking
odometer-based systems to circular system that preserves the factor
and conjugacy structure. To review the definition
of the functor we fix a circular coefficient sequence $\left(k_{n},l_{n}\right)_{n\in\mathbb{N}}$.
Let $\Sigma$ be an alphabet and $\left(\mathtt{W}_{n}\right)_{n\in\mathbb{N}}$
be a construction sequence for an odometer-based system with coefficients
$\left(k_{n}\right)_{n\in\mathbb{N}}$. Then we define a circular
construction sequence $\left(\mathcal{W}_{n}\right)_{n\in\mathbb{N}}$
and bijections $c_{n}:\mathtt{W}_{n}\to\mathcal{W}_{n}$ by induction:
\begin{itemize}
\item Let $\mathcal{W}_{0}=\Sigma$ and $c_{0}$ be the identity map.
\item Suppose that $\mathtt{W}_{n}$, $\mathcal{W}_{n}$ and $c_{n}$ have
already been defined. Then we define 
\[
\mathcal{W}_{n+1}=\left\{ \mathcal{C}_{n}\left(c_{n}\left(\mathtt{w}_{0}\right),c_{n}\left(\mathtt{w}_{1}\right),\dots,c_{n}\left(\mathtt{w}_{k_{n}-1}\right)\right)\::\:\mathtt{w}_{0}\mathtt{w}_{1}\dots\mathtt{w}_{k_{n}-1}\in\mathtt{W}_{n+1}\right\} 
\]
 and the map $c_{n+1}$ by setting 
\[
c_{n+1}\left(\mathtt{w}_{0}\mathtt{w}_{1}\dots\mathtt{w}_{k_{n}-1}\right)=\mathcal{C}_{n}\left(c_{n}\left(\mathtt{w}_{0}\right),c_{n}\left(\mathtt{w}_{1}\right),\dots,c_{n}\left(\mathtt{w}_{k_{n}-1}\right)\right).
\]
 In particular, the prewords are 
\[
P_{n+1}=\left\{ c_{n}\left(\mathtt{w}_{0}\right)c_{n}\left(\mathtt{w}_{1}\right)\dots c_{n}\left(\mathtt{w}_{k_{n}-1}\right)\::\:\mathtt{w}_{0}\mathtt{w}_{1}\dots\mathtt{w}_{k_{n}-1}\in\mathtt{W}_{n+1}\right\} .
\]
 
\end{itemize}
\begin{definition}
Suppose that $\mathbb{K}$ is built from a construction sequence $\left(\mathtt{W}_{n}\right)_{n\in\mathbb{N}}$
and $\mathbb{K}^{c}$ has the circular construction sequence $\left(\mathcal{W}_{n}\right)_{n\in\mathbb{N}}$
as constructed above. Then we define a map $\mathcal{F}$ from the
set of odometer-based systems (viewed as subshifts) to circular systems
(viewed as subshifts) by 
\[
\mathcal{F}\left(\mathbb{K}\right)=\mathbb{K}^{c}.
\]
 
\end{definition}

\begin{remark}
The map $\mathcal{F}$ is a bijection between odometer-based symbolic
systems with coefficients $\left(k_{n}\right)_{n\in\mathbb{N}}$ and
circular symbolic systems with coefficients $\left(k_{n},l_{n}\right)_{n\in\mathbb{N}}$
that preserves uniformity. Since the construction sequences for our
odometer-based systems will be uniquely readable, the corresponding
circular construction sequences will automatically satisfy the strong
readability assumption.
\end{remark}

In \cite{FW2} it is shown that $\mathcal{F}$ gives an isomorphism between the categories $\mathcal{OB}$ and $\mathcal{CB}$. We state the following fact which is part of the main result in \cite{FW2}.

\begin{fact}\label{fact:functor}
For a fixed circular coefficient sequence $(k_n,l_n)_{n\in \N}$ the categories $\mathcal{OB}$ and $\mathcal{CB}$ are isomorphic by the functor $\mathcal{F}$ that takes synchronous isomorphisms to synchronous isomorphisms and  anti-synchronous isomorphisms to anti-synchronous isomorphisms.
\end{fact}

\section{Abstract untwisted AbC method}\label{section abstract AbC}

We recall the AbC method in a way that is most convenient to us. The exposition is similar to the one in \cite[section 6.1]{FW1} and \cite[Part I section 8]{Ka03}.
 
Let $R$ be the usual action of the circle $\T^1$ on $\T^2$ obtained by translation of the first coordinate. More precisely,
\begin{align}
R_t:\T^2\to\T^2\qquad\text{defined by}\qquad R_t(x_1,x_2)=(x_1+t,x_2)
\end{align}

\subsection{Notations}

We introduce some partitions of the circle $\T^1$. For any natural number $q$ we define the partition of the unit circle into half open intervals of length $1/q$ as follows,
\begin{align}
\mathcal{I}_{q}:=\Big\{\Big[\frac{i}{q},\frac{i+1}{q}\Big)\subset \T:i=0,1,\ldots, q-1\Big\}
\end{align}
Next, given natural numbers $s$ and $q$, we define a partition of the torus $\T^2$ into rectangles of length $1/q$ and height $1/s$ as follows,
\begin{align}
\xi^s_q:=\mathcal{I}_q\otimes\mathcal{I}_s:= \Big\{\Big[\frac{i}{q},\frac{i+1}{q}\Big)\times\Big[\frac{j}{s},\frac{j+1}{s}\Big)\subset \T:i=0,1,\ldots, q-1; j=0,1,\ldots, s-1   \Big\}
\end{align}
In the AbC construction we deal with specific sequences of natural numbers $q_n$ and $s_n$ and we will often use the following notation for convenience,
\begin{align}
\xi_n:=\xi_{q_n,s_n}
\end{align}
and the respective atoms of the above partition will be denoted by 
\begin{align}
R^n_{i,j}  : =  \Big[\frac{i}{q_n},\frac{i+1}{q_n}\Big)\times\Big[\frac{j}{s_n},\frac{j+1}{s_n}\Big)
\end{align}
We note that with $\a=p/q$, $p$ and $q$ mutually prime, the atoms of $\xi_q^s$ is permuted by the action of $R_\a$. This action results in a permutation consisting of $s$ cycles, each of length $q$.

\subsection{The untwisted AbC method}

Now we give a description of the abstract `untwisted' AbC method. This is an inductive process and we assume that the construction has been carried out till the $n$ th stage. So we have the following information available to us,
\begin{enumerate}
\item Sequences of natural numbers $\{k_m\}_{m=1}^{n-1},\{s_m\}_{m=1}^{n},\{l_m\}_{m=1}^{n-1},\{p_m\}_{m=1}^n,\{q_m\}_{m=1}^n$ and a sequence of rational numbers $\{\a_m\}_{m=1}^n$ satisfying the following relations
\begin{align}
p_{m+1}=p_mq_mk_ml_m+1\qquad\quad q_m=k_ml_mq_m^2\qquad\quad\a_m=p_m/q_m\qquad\quad s_m^{k_m}\geq s_{m+1}
\end{align} 
\item Measure preserving transformations $\{h_m\}_{m=1}^n$ of the torus $\T^2$. We assume that $h_m$ is a permutation of $\xi_{k_{m-1}q_{m-1}}^{s_m}$, $h_m$ commutes with $R^{1/q_{m-1}}$ and $h_m$ leaves the rectangle $\cup_{j=0}^{s_{m-1}-1} R_{i,j}^{m-1}$ invariant. \footnote{ The last condition is why we call this version of Anosov-Katok \emph{untwisted}.}
\item Measure preserving transformations $\{H_m\}_{m=1}^n$ and $\{T_m\}_{m=1}^n$ of the torus $\T^2$ satisfying the following conditions,
\begin{align}
H_m:=h_0\circ h_1 \circ\ldots\circ h_m\qquad\qquad T_m:= H_m\circ R_{\a_m}\circ H_m^{-1}
\end{align}
\end{enumerate} 

Now we describe how the construction is carried out in the $n+1$ the stage of the AbC method. We choose a natural number $s_{n+1}$ followed by the natural number $k_n$ so that the following growth condition is satisfied
\begin{align}
s_n^{k_n}\geq s_{n+1}
\end{align}
Next we choose a measure preserving transformation $h_{n+1}$ of $\T^2$ which is a permutation of $\xi_{k_nq_n,s_{n+1}}$ (and hence a permutation of $\xi_{n+1}$). Since we are doing the untwisted version of the AbC method, we additionally ensure that the transformation $h_{n+1}$ leaves the rectangle $[0,1/q_n)\times \T$ invariant. We finally choose $l_n$ to be a large enough natural number so that the conjugacies $T_n$ and $T_{n+1}$ are close enough and the sequence $T_n$ converges in the weak topology. We also note that $p_{n+1},q_{n+1}$ and $\a_{n+1}$ are automatically determined by the formulae $p_{n+1}=p_nq_nk_nl_n+1, q_{n+1}=k_nl_nq_n^2$ and $\a_{n+1}=p_{n+1}/q_{n+1}$. This completes the description of our version of the abstract AbC method.

Now we define the following sequence of partitions,
\begin{align}
\zeta_n = H_n(\xi_n)
\end{align}
and note that since $h_{m'}$ is a permutation of $\xi_{m}$ for all $m'\leq m$, we can conclude that $\zeta_n$ is only a permutation of $\xi_{n}$ and hence is a generating sequence of finite partitions if $\xi_n$ is a generating sequence of partitions. 

\subsection{Special requirements}\label{section requirements}

Till now we described the abstract untwisted AbC method in its full generality. In our case we would need it to satisfy the following additional requirements:

\begin{itemize}
\item \textbf{Requirement 1:} The sequence $s_n$ tends to $\infty$
\item \textbf{Requirement 2:} For each $R^n_{0,j}\in\xi_n$ and each $s<s_{n+1}$, we have
\begin{align}
\Big\{t<k_n:h_{n+1}\Big(\Big[\frac{t}{k_nq_n},\frac{t+1}{k_nq_n}\Big)\times\Big[\frac{s}{s_{n+1}},\frac{s+1}{s_{n+1}}\Big)\Big)\subset R^n_{0,j}\Big\}=\frac{k_n}{s_n}
\end{align} 
Note that this assumption allows us to define a map $s\mapsto (j_0,\ldots,j_{k_n-1})_s$  from $\{0,1,\ldots,s_{n+1}-1\}$ to $\{0,1,\ldots,s_n-1\}^{k_n}$ so that for any fixed $s<s_{n+1}$,
\begin{align*}
h_{n+1}\Big(\Big[\frac{t}{k_nq_n},\frac{t+1}{k_nq_n}\Big)\times\Big[\frac{s}{s_{n+1}},\frac{s+1}{s_{n+1}}\Big)\Big)\subset R^n_{0,j_t}
\end{align*}
\item \textbf{Requirement 3:} We assume that the map $s\mapsto (j_0,\ldots,j_{k_n-1})_s$ is one to one.
\end{itemize} 
Note that the above requirements are more than enough to guarantee ergodicity for the limit transformation $T$. We end this section by stating an obvious lemma which serves as a converse to the above.

\begin{lemma}\label{lemma generating 1}
Let $w_0,\ldots,w_{s_{n+1}-1}\subset \{0,1,\ldots,s_n-1\}^{k_n}$ be words such that each $i$ with $0\leq i\leq s_n$ occurs $k_n/s_n$ times in each $w_j$. Then there is an invertible measure preserving $h_{n+1}$ commuting with $R_{\a_n}$ and inducing a permutation of $\xi_{k_nq_n,s_{n+1}}$ such that if $j_t$ is the $t$ th letter of $w_s$ then 
\begin{align*}
h_{n+1}\Big(\Big[\frac{t}{k_nq_n},\frac{t+1}{k_nq_n}\Big)\times\Big[\frac{s}{s_{n+1}},\frac{s+1}{s_{n+1}}\Big)\Big)\subset R^n_{0,j_t}
\end{align*}
\end{lemma}

\section{Approximating partition permutations by real-analytic diffeomorphisms}\label{section approximation}

The purpose of this section is to show that any permutation of a partition of $\T^2$ by a rectangular grid can be approximated sufficiently well by real-analytic diffeommorphisms. This is the real-analytic version of the content of \cite[section 6.2]{FW1}. We note that their smooth construction can be done on the torus, annulus or the disk. Unfortunately the lack of bump functions in the real-analytic category makes life harder and our real-analytic constructions are only valid for the torus. For a disk, even very basic questions like the existence of real-analytic ergodic diffeomorphisms remain open. One can refer to \cite[section 7.1]{FK} and \cite[section 6.3]{Ba-Ku} for a comprehensive analysis of known difficulties for a real-analytic AbC method on arbitrary real-analytic manifolds.

\subsection{Block-slide type maps and their analytic approximations}

We recall that a \emph{step function} on the unit interval is a finite linear combination of indicator functions on intervals. We define the following two types of piecewise continuous maps on $\T^2$,
\begin{align}
& \mathfrak{h}_1:\T^2\to\T^2\qquad\text{defined by}\qquad\mathfrak{h}_1(x_1,x_2):=(x_1,\; x_2 + s_1(x_1)\mod 1)\\
& \mathfrak{h}_2:\T^2\to\T^2\qquad\text{defined by}\qquad\mathfrak{h}_2(x_1,x_2):=(x_1 + s_2(x_2)\mod 1,\; x_2)
\end{align}
where $s_1$ and $s_2$ are step functions on the unit interval. 
The first map has the same effect as partitioning $\T^2$ into smaller rectangles using vertical lines and sliding those rectangles vertically according to $s_1$. On the other hand the second map has the same effect as partitioning $\T^2$ into smaller rectangles using horizontal lines and sliding those rectangles horizontally according to $s_2$.
We refer to any finite composition of maps of the above kind as a \emph{block-slide} type of map on $\T^2$. This is somewhat similar to playing a \emph{game of nine} without the vacant square on $\T^2$.

The purpose of the section is to demonstrate that a block-slide type of map can be approximated extremely well by measure preserving real analytic diffeomorphisms. This can be achieved because step function and be approximated well by real analytic functions. We have the following lemma where we achieve this approximation and a little more to guarantee periodicity with a pre-specified period.

\begin{lemma} \label{lemma approx}
Let $k$ and $N$ be two positive integer and $\a=(\a_0,\ldots,\a_{k-1})\in [0,1)^k$.  Consider a step function of the form 
\begin{align*}
\tilde{s}_{\a,N}:[0,1)\to \R\quad\text{ defined by}\quad \tilde{s}_{\a,N}(x)=\sum_{i=0}^{kN-1}\tilde{\a}_i\chi_{[\frac{i}{kN},\frac{i+1}{kN})}(x)
\end{align*}
Here $\tilde{\a}_i:=\a_j$ where $j:=i\mod k$. Then, given any $\e>0$ and $\d>0$, there exists a periodic real-analytic function $s_{\a,N}:\R\to\R$ satisfying the following properties:
\begin{enumerate}
\item Entirety: The complexification of $s_{\a,N}$ extends holomorphically to $\C$. 
\item Proximity criterion: $s_{\a,N}$ is $L^1$ close to $\tilde{s}_{\a,N}$. In fact we can say more, \begin{align}\label{nearness}
\sup_{x\in[0,1)\setminus F}|s_{\a,N}(x)-\tilde{s}_{\a,N}(x)|<\e,
\end{align}
where $F=\cup_{i=0}^{kN-1}I_i\subset [0,1)$ is a union of intervals centred around $\frac{i}{kN},\;i=1,\ldots, kN-1$ and $I_0=[0,\frac{\d}{2kN}]\cup[1-\frac{\d}{2kN},1)$ and $\l(I_i)=\frac{\d}{kN}\;\forall\; i$. 
\item Periodicity: $s_{\a,N}$ is $1/N$ periodic. More precisely, the complexification will satisfy,
\begin{align}\label{boundedness} 
s_{\a,N}(z+n/N)=s_{\a,N}(z)\qquad\forall\; z\in\C\text{ and }n\in\Z
\end{align}
\item Bounded derivative: The derivative is small outside a set of small measure,
\begin{align} \label{derivative bound}
\sup_{x\in[0,1)\setminus F}|s_{\a,N}'(x)|<\e 
\end{align}
\end{enumerate}
\end{lemma}

\begin{proof}
See \cite[Lemma 4.7]{Ba-Ns} and \cite[Lemma 3.6]{Kana}.
\end{proof}

Note that the condition \ref{boundedness} in particular implies 
\begin{align*}
\sup_{z: \text{Im}(z)<\rho}s_{\b,N}(z)<\infty\quad\forall\; \rho>0.
\end{align*}
Indeed, for any $\rho>0$, put $\Omega'_\rho=\{z=x+iy:x\in [0,1], |y|<\rho\}$ and note that entirety of $s_{\b,N}$ combined with compactness of $\overline{\Omega'_\rho}$ implies $\sup_{z\in\Omega'_\rho}|s_{\b,N}(z)|<C$ for some constant $C$. Periodicity of $s_{\a,N}$ in the real variable and the observation $\Omega_\rho=\cup_{n\in\Z}\left(\Omega'_\rho + n\right)$ implies that $\sup_{z\in\Omega_\rho}|s_{\b,N}(z)|<C$. We have essentially concluded that  $s_{\b,N}\in C^\omega_\infty(\T^1)$. 

Now we show that block-slide type of maps on $\T^2$ can be approximated well by entirely extendable real-analytic diffeomorphisms.

\begin{proposition} \label{proposition approximation}
Let $\mathfrak{h}:\T^2\to\T^2$ be a block-slide type of map which commutes with $\phi^{1/q}$ for some natural number $q$. Then for any $\e>0$ and $\d>0$ there exists a measure preserving real-analytic diffeomorphim $h\in\text{Diff }^{\omega}_\infty(\T^2,\mu)$ such that the following conditions are satisfied:
\begin{enumerate}
\item Proximity property: There exists a set $E\subset\T^2$ such that $\mu(E)<\d$ and $\sup_{x\in\T^2\setminus E}\|h(x)-\mathfrak{h}(x)\|<\e$. 
\item Commuting property: $h\circ\phi^{1/q}=\phi^{1/q}\circ h$
\end{enumerate} 
In this case we say the the diffeomorphism $h$ is $(\e,\d)$-close to the block-slide type map $\mathfrak{h}$. 
\end{proposition}

\begin{proof}
See \cite[Proposition 2.22]{Ba-Ku}
\end{proof}

\subsection{Approximating partition permutation by diffeomorphisms}

Now we describe how partition permutations of $\T^2$ can be approximated well enough by real-analytic diffeomorphisms. At this juncture we point out that for the purpose of the AbC method we would like the approximating diffeomorphism constructed here to commute with $R^{1/q}$ for some given natural number $q$. In the smooth category this is achieved by carrying out the construction on a fundamental domain of $R^{1/q}$ in such a way that this map is identity near the boundary of this fundamental domain and then we glue together $q$ translated copies of this diffeomorphism and the resulting diffeomorphism commutes with $R^{1/q}$ by construction.

In the real-analytic category we do not know how to reproduce the result in a similar fashion and hence we do all construction on the whole of $\T^2$ rather than a fundamental domain of $\T^2$. The problem if we do the construction this way is that we have to keep track that of commutativity all along our construction.

We briefly describe how this construction can be achieved using notations from section \ref{section abstract AbC}. We fix natural numbers $k,s,q$ and $\Pi$ a permutation of $\xi_{kq}^s$ which commutes with $R^{1/q}$.

In \cite[section 5.1]{Ba-Ku} we proved the subsequent statement on the approximation of arbitrary permutations.

\begin{proposition}[\cite{Ba-Ku}, Theorem E] \label{permutation = block-slide}
Let $k,q,l \in \N$ and $\Pi$ be any permutation of $kql$ elements. We can naturally consider $\Pi$ to be a permutation of the partition $\mathcal{S}_{kq,l}$ of the torus $\T^2$. Assume that $\Pi$ commutes with $\phi^{1/q}$. Then $\Pi$ is a block-slide type of map.
\end{proposition}

We have the approximate real-analytic version of the above theorem as follows: 

\begin{theorem}\label{theorem approx}
Let $\Pi$ be any permutation of $k\times s $ rectangles which partitions $[0,1/q)\times\T^1$. In particular we can extend this $\pi$ to a permutation of $\xi_{kq}^s$ which commutes with $\phi^{1/q}$ (see proposition \ref{permutation = block-slide}). Then for any $\e>0$, there exists a diffeomorphism $h\in\text{Diff }^\omega_\infty(\T^2,\mu)$ such that for a set $L\subset \T^2$, the following conditions are satisfied:
\begin{enumerate}
\item $\mu(L)>1-\e$.
\item For any $x\in L\cap R$, $h(x)\in \Pi(R)$ for any $R\in\xi^s_q$.
\item $\phi^{1/q}\circ h = h\circ\phi^{1/q}$. 
\end{enumerate} 
We say that $h$ $\e$-approximates $\Pi$.
\end{theorem}

\begin{proof}
Follows using proposition \ref{permutation = block-slide} followed by proposition \ref{proposition approximation}.  
\end{proof}

We note that the above theorem is the real-analytic counterpart of \cite[Theorem 35]{FW1} and it is valid for the torus only. This is essentially the main difference of our paper with the Foreman-Weiss work. We also remark that one can further improve upon the above result and obtain the real-analytic analogue of \cite[Theorem 1.2]{AK} and a reduced version of \cite[Theorem 1.1]{AK} without the boundary condition. Of course one already has real-analytic version of these theorems (Moser's theorem is true in the real-analytic category) but the complexification of such diffeomorphisms are not known to be entire in any sense and hence these are not compatible with the AbC method. So the above theorem and the aforementioned generalizations are the best we can hope for at the time. 

\section{Real-analytic AbC method}\label{section real-analytic AbC}

Our objective in this section is to produce examples of real-analytic diffeomorphisms using the AbC method. We will also show that these real-analytic AbC diffeomorphisms we construct here are isomorphic to the abstract AbC transformations constructed earlier.

\subsection{Real-analytic AbC method}

\begin{theorem}[Real-analytic untwisted AbC diffeomorphisms] \label{theorem analytic AbC}
Fix a number $\rho>0$. Suppose $T:\T^2\to\T^2$ is a measure preserving transformation built by the abstract AbC method using parameter sequences $\{k_n\}_{n=1}^\infty$ and $ \{l_n\}_{n=1}^\infty$. 

Then if $ \{l_n\}_{n=1}^\infty$ is a sequence of numbers which grows fast enough (see \ref{fast enough}), there exists a diffeomorphism $T^{(\mathfrak{a})}\in\text{Diff }^\omega_\rho(\T^2,\mu)$  which is measure theoretically isomorphic to $T$.
\end{theorem} 

\begin{proof}
Fix $\{\e_n\}_{n=1}^\infty$ such that $\sum_{n=1}^\infty\e_n<\infty$. Let $\{h_n\}_{n=1}^\infty, \{H_n=h_1\circ\ldots\circ h_n\}_{n=1}^\infty$ and $\{T_n=H_n\circ R^{\a_n}\circ H_n^{-1}\}_{n=1}^\infty$ be a sequence of transformations constructed using parameters $\{k_n\}_{n=1}^\infty$ and $\{l_n\}_{n=1}^\infty$ via the abstract AbC method.

Using theorem \ref{theorem approx} we construct diffeomorphisms $h_n^{(\mathfrak{a})}\in\text{Diff }^\omega_\infty(\T^2,\mu)$ such that $h_n^{(\mathfrak{a})}$ $\e_n$-approximates $h_n$. We put 
\begin{align}
H_n^{(\mathfrak{a})}:=h_1^{(\mathfrak{a})}\circ\ldots\circ h_n^{(\mathfrak{a})}\qquad\qquad T_n^{(\mathfrak{a})}:=H_n^{(\mathfrak{a})}\circ R^{\a_n}\circ (H_n^{(\mathfrak{a})})^{-1}
\end{align}
Now we make the following observation regarding proximity exploiting the commutation relation:
\begin{align*}
& d_\rho\big(T_{n+1}^{(\mathfrak{a})},T_n^{(\mathfrak{a})}\big)\\
&\qquad =d_\rho\big(H_{n}^{(\mathfrak{a})}\circ R^{\a_n}\circ [h_{n+1}^{(\mathfrak{a})}\circ R^{1/(k_nl_nq_n^2)}\circ (h_{n+1}^{(\mathfrak{a})})^{-1}]\circ (H_n^{(\mathfrak{a})})^{-1}, H_n^{(\mathfrak{a})}\circ R^{\a_n}\circ (H_n^{(\mathfrak{a})})^{-1}\big)
\end{align*}
Recall that $l_n$ is chosen last in the induction step. So, if one choses $l_n$ to be a large enough natural number then from the continuity of $d_\rho$ with respect to composition we obtain $d_\rho(T_{n+1}^{(\mathfrak{a})},T_n^{(\mathfrak{a})})<\e_n$. Since we are dealing with real-analytic functions which are often more delicate than smooth functions, we make some more observations to justify this claim. Note that since $(H_{n+1}^{(\mathfrak{a})})^{-1}\in \text{Diff }^{\omega}_\infty(\T^2,\mu)$, we can choose some $\rho'>\rho$ such that $(H_{n+1}^{(\mathfrak{a})})^{-1}(\Omega_{\rho})\subset\Omega_{\rho'}$. For any $x\in \Omega_\rho$ we put   $y=(H_{n+1}^{(\mathfrak{a})})^{-1}(x)$. So $y,\; R^{1/(k_nl_nq_n^2)}(y) \in \Omega_{\rho'}$. Also note that $H_{n+1}^{(\mathfrak{a})}\circ R^{\a_n}\in  \text{Diff }^{\omega}_\infty(\T^2,\mu)$ and any function in $ \text{Diff }^{\omega}_\infty(\T^2,\mu)$ is uniformly continuous on $\Omega_{\rho'}$. \footnote{ Indeed, if $ F\in  \text{Diff }^{\omega}_\infty(\T^2,\mu)$, then since $F$ is $\Z^2$ periodic, we have for any $\rho''>0 $, $u,v \in\{z=(z_1,z_2)\in\C^2: \text{Re}(z_i)\in [0,1], \; \text{Im}(z_i)\leq \rho''\}$ and $\e>0$, there exists a $\d>0$ independent of $u,v$ such that $\|u-v\|<\d\Rightarrow \|F(u)-F(v)\|<\e$ (this follows from the compactness of the domain). In other words $F$ is uniformly continuous on this restricted domain. For arbitrary $u,v\in\Omega_{\rho''}$, is $\delta$ is small, we can find integers $n_1,n_2, m_1,m_2$ such that $u-(n_1,n_2),v-(m_1,m_2)\in \{z=(z_1,z_2)\in\C^2: \text{Re}(z_i)\in [0,1], \; \text{Im}(z_i)\leq \rho''\}$. And since $F$ is $\Z^2$ periodic, $\|F(u)-F(v)\|= \|F(u-(n_1,n_2))-F(v-(m_1,m_2))\|$. } This implies 
\begin{align}\label{remaark 39}
\tilde{d}_{\rho}(T_{n+1}^{(\mathfrak{a})},T_n^{(\mathfrak{a})}) \leq &\;  \tilde{d}_{\rho'}(H_{n+1}^{(\mathfrak{a})}\circ R^{\a_n}\circ R^{1/(k_nl_nq_n^2)}, H_{n+1}^{(\mathfrak{a})}\circ R^{\a_n})\nonumber\\
\leq  & \sup_{y\in\Omega_{\rho'}}\|H_{n+1}^{(\mathfrak{a})}\circ R^{\a_n}( R^{1/(k_nl_nq_n^2)}(y)), H_{n+1}^{(\mathfrak{a})}\circ R^{\a_n}(y)\|\nonumber\\
< & \; \e_n/4
\end{align}
for $l_n$ sufficiently large. A similar consideration with the inverse gives us the result for $d_\rho$. Hence the sequence $T_n^{(\mathfrak{a})}$ is Cauchy and converges to some $T^{(\mathfrak{a})}\in \text{Diff }^\omega_\rho(\T^2,\mu)$.

Next we need to prove that $T^{(\mathfrak{a})}$ is in fact measure theoretically isomorphic to $T$. Our plan is to use lemma \ref{lemma isomorphism} for the proof.

In the language of the lemma, we put $(\mathbb{X},\mathcal{B},m)=(\mathbb{X}',\mathcal{B}',m')=(\T^2,\mathcal{B},\mu)$. Next we define $K_n$. We put 
\begin{align}
K_n:\T^2\to\T^2\qquad\text{defined by}\qquad K_n:= H_n^{(\mathfrak{a})}\circ H_n^{-1}
\end{align}
From the definition it follows that $K_n$ is a isomorphism between $T_n$ and $T_n^{(\mathfrak{a})}$. 

We define the two sequences of partitions $\mathcal{P}_n:=\zeta_n=H_n(\xi_n)$ and $\mathcal{P}_n':=K_n(\zeta_n)=H_n^{(\mathfrak{a})}(\xi_n)$ and observe that using lemma \ref{lemma generating 1} we can conclude that $\{\mathcal{P}_n\}^{\infty}_{n=1}$ is generating. We have to work a little to show that $\mathcal{P}_n'$ is generating.

We define the set $L_{n}$ to be the set of measure more than $1-\e_n$ corresponding to $h_n$ as per Theorem \ref{theorem approx}. Consider the following sequence of sets:
\begin{align*}
G_n\coloneqq L_n\cap \bigcap_{m=n+1}^{\infty} (h_{n+1}^{(\mathfrak{a})}\circ\ldots\circ h_{m}^{(\mathfrak{a})})^{-1}(L_m)
\end{align*}
Note that $G_n$ is an increasing sequence and the Borel-Cantelli lemma guarantees that $\mu(G_n)\nearrow 1$. 

We pick a measurable $D\subset \T^2$ and $\d>0$. There exists some $n_0$ such that $\mu(G_m)>1-\frac{\d}{2}$ for all $m>n_0$. 

We put $D'=(H^{(\mathfrak{a})}_{n_0})^{-1}(D)$ and since $\{\xi_n\}^{\infty}_{n=1}$ is a generating sequence, there exists an $m>n_0$ and a collection $\mathcal{C}_m'\subset \xi_m$ such that 
\begin{align*}
& \mu\big((\bigcup_{C\in\mathcal{C}_m'} C)\triangle D'\big)<\d/2\\
\Rightarrow & \mu\big((\bigcup_{C\in\mathcal{C}_m'} H_{n_0}^{(\mathfrak{a})}(C))\triangle D\big)<\d/2
\end{align*}
On the other hand, note that for any $m>n_0$
\begin{align*}
& \mu\big(\bigcup_{R\in\xi_m}  h_{n_0+1}^{(\mathfrak{a})}\circ\ldots\circ h_{m}^{(\mathfrak{a})}(R) \;\triangle\; h_{n_0+1}\circ\ldots\circ h_m(R)\big)<\delta/2\\
\Rightarrow \; & \mu\big(\bigcup_{R\in\xi_m} (H_{n_0}^{(\mathfrak{a})}\circ h_{n_0+1}^{(\mathfrak{a})}\circ\ldots\circ h_{m}^{(\mathfrak{a})}(R) \;\triangle\; H_{n_0}^{(\mathfrak{a})}\circ h_{n_0+1}\circ\ldots\circ h_{m}(R)\big)<\delta/2\\
\Rightarrow \; & \mu\big(\bigcup_{R\in\xi_m} (H_{m}^{(\mathfrak{a})} (R) \;\triangle\; H_{n_0}^{(\mathfrak{a})}\circ h_{n_0+1}\circ\ldots\circ h_{m}(R)\big)<\delta/2\\
\Rightarrow \; & \mu\big(\bigcup_{R\in\xi_m} (H_{m}^{(\mathfrak{a})} (R) \;\triangle\; H_{n_0}^{(\mathfrak{a})} (\Pi(R))\big)<\delta/2
\end{align*}
Where $\Pi\coloneqq h_{n_0+1}\circ\ldots\circ h_{m}$ is a permutation of $\xi_m$. So, 
\begin{align*}
& \mu\big(\bigcup_{R\in\Pi^{-1}(\mathcal{C}_m')} H_{m}^{(\mathfrak{a})} (R) \;\triangle\; D\big) \\
\leq & \; \mu\big(\bigcup_{R\in\Pi^{-1}(\mathcal{C}_m')} H_{m}^{(\mathfrak{a})} (R) \;\triangle\; H_{n_0}^{(\mathfrak{a})} (\Pi(R))\big) + \mu\big(\bigcup_{R\in\Pi^{-1}(\mathcal{C}_m')} H_{n_0}^{(\mathfrak{a})} (\Pi(R)) \;\triangle\; D\big)\\
\leq & \; \mu\big(\bigcup_{R\in\xi_m} H_{m}^{(\mathfrak{a})} (R) \;\triangle\; H_{n_0}^{(\mathfrak{a})} (\Pi(R))\big) + \mu\big(\bigcup_{C\in\mathcal{C}_m'} H_{n_0}^{(\mathfrak{a})} (C) \;\triangle\; D\big)\\
\leq & \; \d/2 +\d/2 \\
= & \; \d
\end{align*}
This shows that $\mathcal{P}_n'$ is a generating sequence of partitions.

Our next objective is to show that $D_\mu(K_{n+1}(\mathcal{P}_n),K_n(\mathcal{P}_n))<\e_n$. So we do the following computations:
\begin{align*}
K_n(\mathcal{P}_n)= & \; K_n(\zeta_n)
=  \; H_n^{(\mathfrak{a})}\circ H_n^{-1}(H_n(\xi_n))
=  \; H_n^{(\mathfrak{a})}(\xi_n)
=  \; H_{n}^{(\mathfrak{a})}\circ h_{n+1}\circ h_{n+1}^{-1}(\xi_n)
\end{align*}
On the other hand, 
\begin{align*}
K_{n+1}(\mathcal{P}_n)= & \; K_{n+1}(\zeta_n)
= \; H_{n+1}^{(\mathfrak{a})}\circ H_{n+1}^{-1}(H_n(\xi_n))
= \; H_{n}^{(\mathfrak{a})}\circ h_{n+1}^{(\mathfrak{a})}\circ h_{n+1}^{-1}(\xi_n)
\end{align*}
Put $\mathcal{Q}_n\coloneqq h_{n+1}^{-1}(\xi_n)$ and note that by construction $h_{n+1}^{(\mathfrak{a})}(\mathcal{Q}_n)$ approximates $h_{n+1}(\mathcal{Q}_n)$ and we are done.
\end{proof}

\subsection{Fast enough in the real-analytic context}

We investigate what it means to be \emph{fast enough} in theorem \ref{theorem analytic AbC}. We fix a sequence $\{\e_n\}_{n=1}^\infty$ and assume that for any $n\in \N$:
\begin{align} \label{eq:epsdecrease}
    \e_n/4>\sum_{m=n}^\infty \e_m
\end{align}

For each choice of sequences $\{k_n\}_{m=1}^{n}$, $\{l_n\}_{m=1}^{n-1}$ and $\{s_n\}_{m=1}^{n+1}$ of natural numbers, we can have finitely many permutations of $\xi_{k_nq_n}^{s_{n+1}}$ and hence finitely many choices of $h_{n+1}$. For each such choice, there exists a natural number $l_n:=l_n(h_{n+1}, \{k_n\}_{m=1}^{n},\{l_n\}_{m=1}^{n-1},\{s_n\}_{m=1}^{n+1}, \rho)$ such that for any $l\geq l_n$, we can choose $h_{n+1}^{(\mathfrak{a})}$ such that 
\begin{align}\label{equation 26}
d_\rho(T_n^{(\mathfrak{a})},T_{n+1}^{(\mathfrak{a})})<\e_n/4
\end{align}
Finally to get an uniform estimate, we put 
\begin{align}
l_n^*=l_n^*(\{k_n\}_{m=1}^{n},\{l_n\}_{m=1}^{n-1},\{s_n\}_{m=1}^{n+1}, \rho)\coloneqq \max_{h_{n+1}}l_n(h_{n+1}, \{k_n\}_{m=1}^{n},\{l_n\}_{m=1}^{n-1},\{s_n\}_{m=1}^{n+1}, \rho)
\end{align}
Now we note that for our construction there is a relation $s_{n}=s_{n-1}^{k_{n-1}}$. We can define $b_n=b_n(\{k_n\}_{m=1}^{n-1})$ such that $s_{n}<b_n$. Thus we can have our sequence $l_n^*$ depend only on $k_n$ s (after speeding up a little if needed).

In conclusion we obtain a sequence of natural numbers 
\begin{align}
l_n^*=l_n^*(\{k_n\}_{m=1}^{n},\{l_n\}_{m=1}^{n-1},\rho)
\end{align}
and we say a sequence of natural numbers $\{l_n\}_{n=1}^\infty$ grows \emph{fast enough} if 
\begin{align}\label{fast enough}
l_n\geq l_n^*\qquad \forall\; n.  
\end{align}

\section{Symbolic representation of AbC systems}

The purpose of this section is establish the main result where a symbolic representation of untwisted real-analytic AbC diffeomorphisms is obtained. This section is almost identical to \cite[section 7]{FW1} though our presentation is much succinct.

We begin this section with the intrinsic description of obtaining a symbolic representation of the factor of a periodic process using construction sequences.

\subsection{Symbolic representation of periodic processes}

Here we exhibit how a periodic process can be viewed as a symbolic system. Everything in this section is done on a standard measure space $(\mathbb{X},\mathcal{B},m)$.

Let $\{\e_n\}_{n\in\N}$ be a summable sequence of positive numbers. Let $\{(\tau_n,\mathcal{P}_n)\}$ be a sequence of periodic processes converging to some transformation $T$. We assume that the height of all the towers of $(\tau_n,\mathcal{P}_n)$ is $q_n$ and $(\tau_{n+1},\mathcal{P}_{n+1})$ $\e_n$-approximates $(\tau_n,\mathcal{P}_n)$ with error set $D_n$. In addition we can assume (after removing some steps in the beginning if needed) that $\mu(\cup D_n)<1/2$ and also assume $q_0=1$. 

We put $G_n=\mathbb{X}\setminus \cup_{m\geq n}D_n$. Then note that $\{G_n\}_{n\in\N}$ is an increasing sequence of sets with $\mu(G_n)\nearrow 1$. Also, when restricted to $G_n$, $\{\mathcal{P}_m\}_{m\geq n}$ is a decreasing sequence of partitions.  For each tower of $(\tau_n,\mathcal{P}_n)$, the measure of the intersection of $G_n$ with each level of this tower are all the same.

Let $\{\mathcal{T}_i\}_{i=0}^{s_0-1}$ be the towers of $(\tau_0,\mathcal{P}_0\}$. Since $q_0=1$, each $\mathcal{T}_i$ contains a single set and we define $\mathcal{Q}_0$ to be the collection of all these sets. 

We inductively define two sequence of sets $\{B_n\}_{n\in\N}$ and $\{E_n\}_{n\in\N}$. Let $B_0=E_0=\emptyset$. Next we assume that we have defined $\{B_m\}_{m=0}^{n}$ and $\{E_m\}_{m=0}^{n}$. \\
At the $n+1$ th stage of the induction process we have to define $B_{n+1}$ and $E_{n+1}$. Since $\tau_{n+1}$ $\e_n$-approximates $(\tau_n,\mathcal{P}_n)$, we note that each tower of $(\tau_{n+1},\mathcal{P}_{n+1})$ when restricted to $G_{n+1}$ contains 
\begin{enumerate}
\item Contiguous levels: Contiguous sequences of levels of length $q_n$ contained in towers of $(\tau_n,\mathcal{P}_n)$ 
\item Interspersed levels: Levels of the towers of $(\tau_{n+1},\mathcal{P}_{n+1})$ intersected with $G_{n+1}$ not in the above contiguous sequences.
\end{enumerate}
Next we divide each of the maximal contiguous portions of the interspersed levels into two contiguous portions arbitrarily. We define $E_{n+1}$ to be the union of $E_n$  and the subcollection of levels that comes first and $B_{n+1}$ to be the union of $B_n$ and the subcollection that comes second. Also we note the top and the bottom contiguous subcollection of interspersed levels are joined together and is considered to be a single contiguous subcollection for this process. \\
This completes the construction of the two sequences and we note that $G_n = \mathcal{Q}_0\cup B_n\cup E_n$.  

Our next goal is to find a symbolic representation of a \emph{factor} of the limiting transformation $T$ arising from the partition $\mathcal{Q}_0\cup\{\cup_{n\in\N} B_n,\cup_{n\in\N} E_n\}$. We do this in two ways: First we describe the standard procedure using $T$ and then we describe the procedure that uses the aforementioned sequences and construction sequence description of symbolic systems.

Let $\Sigma$ be an alphabet of size $s_0$. Let $b$ and $e$ be two letters not contained in $\Sigma$. The letters  in $\Sigma\coloneqq\{a_i\}_{i=0}^{s_0-1}$ are considered to be indexes for the elements of $\mathcal{Q}_0\coloneqq \{A_i\}_{i=0}^{s_0-1}$. Our symbolic systems is built on the alphabet $\Sigma\cup\{b,e\}$. We define a factor map 
\begin{align}
K:\mathbb{X}\to (\Sigma\cup\{b,e\})^\Z\qquad\text{defined by}\qquad K(x)(i)\coloneqq \begin{cases} a_j & \text{if }\quad T^i(x)\in A_j\\ b& \text{if }\quad T^i(x)\in \cup_{n\in\N}B_n\\ e & \text{if }\quad T^i(x)\in \cup_{n\in\N}E_n\end{cases}
\end{align}   

We give an alternate description of this symbolic system using construction sequences. We inductively define $\mathcal{W}_n$ which are collections of words  of length $q_n$  and surjections $K_n:\{\mathcal{T}\cap G_n:\mathcal{T}$ is a tower of $(\tau_n,\mathcal{P}_n)$ and $G_n\cap\mathcal{T}\neq \emptyset\}\to\mathcal{W}_n$.

First we put $\mathcal{W}_0\coloneqq \Sigma$ and we assume that we have carried out the construction upto the $n$ th stage. At the $n+1$ th stage we define the collection of words $\mathcal{W}_{n+1}$ and $K_{n+1}$. Let $\mathcal{T}$ be a tower of $(\tau_{n+1},\mathcal{P}_{n+1})$. Then with $\mathcal{T}\cap G_{n+1}$ we associate a word $w$ of length $q_{n+1}$ satisfying:
\begin{enumerate}
\item $w(j)=v$ if the $j$-th level of $\mathcal{T}$ is a subset of the  $k$-th level of of a tower $\mathcal{S}$ of $(\tau_n,\mathcal{P}_n)$ and the $k$-th letter of $K_n(\mathcal{S})$ is $v$.
\item $w(j)=b$ if the $j$-th level of $\mathcal{T}$ is a subset of $B_{n+1}\setminus B_n$.  
\item $w(j)=e$ if the $j$-th level of $\mathcal{T}$ is a subset of $E_{n+1}\setminus E_n$.
\end{enumerate}
We define $\mathbb{K}$ to be the collection of all $x\in(\Sigma\cup\{b,e\})^\Z$ such that every contiguous subword of $x$ is a contiguous subword of some $w\in\mathcal{W}_n$ for some $n$. Then $\mathbb{K}$ is a closed shift invariant set that constitutes the support of $K^*\mu$ and it is the required symbolic representation of the factor of $T$ described explicitly earlier (see lemma \ref{lemma 11}). 

We note that in the case of the untwisted AbC transformations we consider satisfying requirements 1, 2 and 3 (see section \ref{section requirements}), $\mathcal{Q}$ will generate the transformation and the resulting symbolic representation will be isomorphic to $T$.

\subsection{Comparison of two periodic processes using a transect - I}\label{section transects 1}

Instead of directly comparing two subsequent periodic processes in the AbC method, we start our study slowly with the study of two periodic processes on the circle. This study will shed light on how a tower in the periodic process at the $n$-th stage of the AbC method compares with a tower at the $n+1$-th stage. In fact we wish to study how the levels of a periodic process on the circle traverses the levels of another periodic process which it $\e$-approximates. The parameters for the processes are chosen similar to the AbC method.

Let $p,q$ be two natural numbers and $\a=p/q$. $\mathcal{I}_q\coloneqq \{I^q_i\coloneqq [i/q,(i+1)/q)\}_{i=0}^{q-1}$ is the standard partition of $[0,1)$ into $q$ equal half open intervals. Recall that the atoms of $\mathcal{I}_q$ have two natural orderings:
\begin{enumerate}
\item The \emph{geometric ordering} is the natural ordering of the intervals i.e. $I_0^q<I_1^q<\ldots<I_{q-1}^q$.
\item The \emph{dynamical ordering} is the ordering that comes from iteration by $R_\a$ i.e. $I_0^q < R^\a(I_0^q)<\ldots<(R^\a)^{q-1}(I_0^q)$.
\end{enumerate}
Recall that with $j_i=(p)^{-1}i \mod q$, the $i$-th interval in the geometric ordering is the $j_i$ th interval in the dynamical ordering. Indeed if $(R^\a)^{j_i}(I_0^q)=I_{i'}^q$, then $pj_i\equiv i'$.

Let $\a=p/q$ and $\a'=p'/q'$ where $p,p',q$ and $q'$ are natural numbers and $q'=klq^2$. We compare the two periodic processes $(R^\a,\mathcal{I}_q)$ and $(R^{\a'},\mathcal{I}_{q'})$.

So if $J$ is a subinterval of $I_t^q$, and if $J$ is not the last subinterval of $I_t^q$ then $R^{\a'}(J)$ is a subinterval of the $j_i$th subinterval in the dynamical ordering i.e. $R^{\a}(I_t^q)$. If $J$ is the last subinterval, then $R^{\a'}(J)$ is geometrically the first subinterval of the $j_{i+1}$-th subinterval in the dynamical ordering.
    
With the above observation in mind, we make detailed analysis of how the interval $J\coloneqq [0,1/q')$ traverses the levels of the tower of $(R^\a,\mathcal{I}_q)$ under the action of $R^{\a'}$. Note that we have to study $q'$ iterates of $R^{\a'}$ to get a complete picture. So we divide the set $\{0,1,\ldots,q'=klq^2\}$ into contiguous portions of size $q$. 

On the first portion i.e. for $n=0,1,\ldots, klq-1$, $(R^{\a'})^n(J)\subset (R^{\a})^n(I_0^q)$ and $(R^{\a'})^{klq}(J)$ is geometrically the first subinterval of $I_1^q$. 

More generally when we study the iterations of $J$ for $n=mklq,\ldots, (m+1)klq$, we see that the iterations exhibit the following three types of behavior:
\begin{enumerate}
\item The \emph{beginning interval} $[mklq,\ldots,mklq+q-j_m)$: This is of length $q-j_m$. Note that $(R^{\a'})^{mkql}(J)$ is geometrically the first subinterval of $I_m^q$. Then with $n$ in the beginning interval, $(R^{\a'})^n(J)$ traverses the interval in places $j_m,j_m+1,\ldots,q-1$ in the dynamical ordering of $(R^\a,\mathcal{I}_q)$.
\item The \emph{middle  interval} $[mkql+q-j_m, mkql+q-j_m+(kl-1)q)$: This is of length $klq-q$. With $n$ in the middle interval, $(R^{\a'})^n(J)$ traverses the intervals following the dynamical ordering of $(R^\a,\mathcal{I}_q)$ starting from $I_0^q$.
\item The \emph{end interval} $[mkql+q-j_m+(kl-1)q, (m+1)kql)$: This has length $j_m$. With $n$ in the end interval, $(R^{\a'})^n(J)$ traverses the interval in places $0,1,\ldots,j_m-1$ in the dynamical ordering of $(R^\a,\mathcal{I}_q)$. $(R^{\a'})^{mklq-1}(J)$ is geometrically the last subinterval of $I^q_m$.
\end{enumerate}    

\subsection{Comparison of two periodic processes using a transect - II}\label{section transects 2}

The previous section did shed some light on how two periodic processes compare on the circle when one $\e$ approximates the other but we note that in the previous section we did not pay much attention to the fact that the AbC method uses the partition whose projection is $\mathcal{I}_{kq}$ and not $\mathcal{I}_q$. We do the study again, but keeping this finer partition in mind.

Let $p,q,p',q',\a$ and $\a'$ be as before. We divide the subinterval $\mathcal{I}_{kq}$ into $k$ ordered sets described as follows:
\begin{align}
w_j\coloneqq\{I_{j+tk}^{kq}:t=0,\ldots, q-1\}
\end{align}
So each $w_j$ is the orbit of $[j/(kq),(j+1)/(kq))$ under $R_\a$. It can be viewed as a word of length $q$ in the alphabet $\mathcal{I}_{kq}$.

Let $J$ be the subinterval as before (see section \ref{section transects 1}). We track the $R_{\a'}$ iterates of $J$ through the $w_j$ s as follows:
\begin{enumerate}

\item[0.] \emph{$0$-th interval} $n\in[0, klq)$: Note that any such $n$ can be written as $n=mlq+sq+t$ for some appropriately chosen $0\leq m<k,0\leq s<l,0\leq t<q$. So $(R_{\a'})^n(J)$ is a subinterval of the $t$-th element of $w_m$. So 
\begin{itemize}
\item $t\in[0,lq)$: The $\mathcal{I}_{kq}$-name agrees with the $w_0$ name repeated $l$ times
\item $t\in[lq,2lq)$: The interval crosses the boundary for the $\mathcal{I}_{kq}$ partition and the name changes to $w_1$ repeated $l$-times.
\item[.] $\ldots\ldots$
\item[]  $k-3$ more times to a total of $k-1$ times.
\item[.] $\ldots\ldots$
\item $t\in [(k-1)lq,klq)$: The interval crosses the boundary for the $\mathcal{I}_{kq}$ partition and the name changes to $w_{k-1}$ repeated $l$-times. $J_{i_{klq-1}}$ is geometrically the last subinterval of $I_1$. 
\end{itemize}
Hence the first $klq$ letters of the $\mathcal{I}_{kq}$-name of any point in $J$ is 
\begin{align}
w_0^lw_1^lw_2^l\ldots w_{k-1}^l
\end{align}  

\item[1.] \emph{$1$-st interval} $n\in[klq,2klq)$: 
\begin{itemize}
\item $t\in [klq,(k+1)lq)$: 
\begin{itemize}
\item $t\in [klq, klq+q-j_1)$: $(R^{\a'})^{klq}(J)=J_{i_{klq}}$ is geometrically the first subinterval of $I_2^q$. we use another $q-j_1$ applications of $R^{\a'}$ to bring $J$ inside $I_0^q$.  
\item $t\in [klq+q-j_1, (k+1)lq-j_1)$: $J_{i_{klq+q-j_1}}$ is a subinterval of the geometrically first subinterval of $\mathcal{I}_{kq}$. In fact it is in the first subinterval of $w_1$. We apply  $R^{\a'}$ $(l-1)q$ times to carry it through $l-1$ copies of $w_1$ and it arrives at $I_0^q$ 
\item $t\in [(k+1)lq-j_1, (k+1)lq)$: We apply  $R^{\a'}$ again $j_1$ times to bring it back to $I_1^q$ inside $w_2$.
\end{itemize}
Resulting name is $b_0^{q-j_1}w_0^{l-1}e_0^{j_1}$. Here $b_j^{q-j_1}$ and $e_j^{j_1}$ are the last $q-j_1$ and the first $j_1$ elements of $w_j$.

\item $t\in [(k+1)lq,(k+2)lq)$: With an identical argument we get the name $b^{q-j_1}_1w_{1}^{l-1}e^{j_1}_1$.

\item[.] $\ldots\ldots$
\item[] $k-3$ more times to a total of $k-1$ times.
\item[.] $\ldots\ldots$
\item $t\in [(k+k-1)lq,2klq)$: With an identical argument we get the name $b^{q-j_1}_{k-1}w_{k-1}^{l-1}e^{j_1}_{k-1}$. $J_{i_{2klq-1}}$ is the geometrically last subinterval of $I_1$.
\end{itemize}
So the $klq$ to $2klq-1$-th letters of the $\mathcal{I}_{kq}$ name of any point in $J$ is given by 
\begin{align}
b_0^{q-j_1}w_0^{(l-1)}e_0^{j_1}b_1^{q-j_1}w_1^{(l-1)}e_1^{j_1}\ldots b_{k-1}^{q-j_1}w_{k-1}^{(l-1)}e_{k-1}^{j_1}
\end{align} 
 
\item[2.] $\ldots$
\item[3.] $\ldots$
\item[] $\ldots$
\item[] $\ldots$

\item[m.] \emph{$m$-th interval} $n\in[mklq,(m+1)klq)$: The $mklq$ to $(m+1)klq-1$-th letters of the $\mathcal{I}_{kq}$ name of any point in $J$ is given by 
\begin{align}
b_0^{q-j_m}w_0^{(l-1)}e_0^{j_m}b_1^{q-j_m}w_1^{(l-1)}e_1^{j_m}\ldots b_{k-1}^{q-j_m}w_{k-1}^{(l-1)}e_{k-1}^{j_m}
\end{align} 

\item[m+1.] $\ldots$
\item[] $\ldots$
\item[] $\ldots$

\item[q-1.] (q-1)-th interval $t\in [(q-1)klq,klq^2)$: An identical argument yields the name
\begin{align}
b^{q-j_{q-1}}w_{0}^{l-1}e^{j_{q-1}} b^{q-j_{q-1}}w_{1}^{l-1}e^{j_{q-1}} \ldots b^{q-j_{q-1}}w_{k-1}^{l-1}e^{j_{q-1}}
\end{align}
\end{enumerate}
So in conclusion any point in $J$ has a $\mathcal{I}_{kq}$-name as follows:
\begin{align}
w=\prod_{i=0}^{q-1}\prod_{j=0}^{k-1}(b_j^{q-j_i}w_j^{l-1}e_j^{j_i})
\end{align}
So our periodic process is isomorphic to the symbolic system defined by the above circular operator.

\subsection{Back to the regular sequence of periodic processes}

Note that the partition $\xi_n$ divides $\T^2$ into $s_n$ identical towers. On each tower , the action is identical and hence when restricted to a single tower of $\mathcal{I}_{k_nq_n}\otimes\mathcal{I}_{s_{n+1}}$ we get the same analysis as before. So we can also copy the previous labeling to a labeling here. 

\subsection{Symbolic representation of AbC systems}

First we describe the general idea behind the representation of the AbC method as a symbolic system. We recall some relevant portions from the abstract untwisted AbC method described in section \ref{section abstract AbC}.  The limit transformation $T$ is obtained as the weak limit of transformations $T_n\coloneqq H_n\circ R^{\a_n}\circ H_n^{-1}$.  At the $n$ th stage $T_n$ permutes the partition $\zeta_n=H_n(\xi_n)$ \footnote{ So $\zeta_n$ is just $\xi_n$ ordered in a different way.}.With the above in mind, we recall that $(\tau_n,\zeta_n)$ is a periodic process where $\tau_n$ is the permutation induced by $T_n$ on $\zeta_n$.  

If $\mathcal{Q}$ is a partition refined by the levels of the towers of a periodic process $\tau$, then the $\mathcal{Q}$ names of any pointwise realization of $\tau$ are constant on the levels of the tower. Hence this is equivalent to naming the levels in the action of $\tau$ on various towers. We call the resulting collection of names the $(\tau,\mathcal{Q})$ names.

Let $\mathcal{Q}^*$ be an arbitrary partition of $\T^2$ refined by $\zeta_n$. We would like to compare the $\mathcal{Q}^*$ names of points under $\tau_n$ and $\tau_{n+1}$. We introduce the partition $(H_n)^{-1}(\mathcal{Q}^*)$. So the problem of finding the $(\tau_{n+1},\mathcal{Q}^*)$ name is equivalent to finding the $(h_{n+1}\circ R^{\a_n}\circ h_{n+1}^{-1}, \mathcal{P})$ names of towers whose levels consists of the partition $(H_n)^{-1}(\zeta_{n+1})=\xi_{n+1}$.

\subsubsection*{Tracking movement of individual rectangles}

For notational simplicity we put $k_n=k, q_n=q, I_i=I_i^{q_n}$ and $J_i=I_i^{q_{n+1}}$. Fix $R\in \xi_{n+1}$. 

Assume that $R_{i,j}^{n+1}=h_{n+1}^{-1}(R)\subset J_i$ for some $i$ and $I_i$ is not the last subinterval of an interval in the partition $\mathcal{I}_{kq}$. In this case both $R^{\a_n}(R_{i,j}^{n+1})$ and $R^{\a_{n+1}}(R_{i,j}^{n+1})$ belong to the same element of $\mathcal{I}_{kq}\times\T^1$.

Since $R^{\a_n}$ commutes with $h_{n+1}$ and $h_{n+1}$ permutes the atoms of $\mathcal{I}_{k_nq_n}\otimes\mathcal{I}_{s_{n+1}}$, we have $h_{n+1}\circ R^{\a_{n+1}}\circ h_{n+1}^{-1}(R)= h_{n+1}\circ R^{\a_{n+1}}(R_{i,j}^{n+1})$ and $h_{n+1}\circ R^{\a_{n}}\circ h_{n+1}^{-1}(R)= R^{\a_{n}}(R)$ belong to the same atom of $\mathcal{I}_{k_nq_n}\otimes\mathcal{I}_{s_{n+1}}$. So they share the same $\mathcal{P}$-name.

On the other hand if $J_i$ is the last subinterval of an interval in the partition $\mathcal{I}_{kq}$, then $R^{\a_{n+1}}$ sends $R^{n+1}_{i,j}$ to the geometrically first subrectangle of a new element $R'$ of $\mathcal{I}_{kq}\otimes\mathcal{I}_{s_{n+1}}$. Thus $h_{n+1}\circ R^{\a_{n+1}}\circ h_{n+1}^{-1}(R)$ is a subrectangle of $h_{n+1}(R')$.

\subsubsection*{Tracking movement of levels through a tower}

We note that the base of the towers of $(\tau_{n+1},\zeta_{n+1})$ are the rectangles $\{H_{n+1}(R^{n+1}_{0,j})\}_{j=0}^s$ while those for the towers of $(h_{n+1}\circ R^{\a_{n+1}}\circ h_{n+1}^{-1},\mathcal{P})$ are $\{h_{n+1}(R^{n+1}_{0,j})\}_{j=0}^s$.

So with $F_0\coloneqq h_{n+1}(R^{n+1}_{0,j})$ as the base, we define $F_t\coloneqq (h_{n+1}\circ R^{\a_{n+1}}\circ h_{n+1}^{-1})^t(h_{n+1}(R^{n+1}_{0,j}))= (h_{n+1}\circ (R^{\a_{n+1}})^t)(R^{n+1}_{0,j}))= h_{n+1}(R^{n+1}_{i_t,j})$ where $i_t$ is the dynamical ordering of the $t$-th interval under iterations by $R^{\a_{n+1}}$. 

We recall the labeling we used in section \ref{section transects 2} using the $w,b$ and $e$ s. So if $t$ is such that $J_{i_t}$ is labeled with a part of $w$, the two transformations $h_{n+1}\circ  R^{\a_{n+1}}\circ h_{n+1}$ and $R_{\a_n}$ move $F_t$ to a subrectangle of the same element of $\mathcal{I}_{kq}\otimes\mathcal{I}_{s_{n+1}}$ and hence the same element of $\mathcal{P}$.

For $j < k, t<q$ and $s<s_{n+1}$, define:
\begin{align}
R_{j,t,s}\coloneqq \big[\frac{j+tk}{kq},\frac{j+tk+1}{kq}\big)\times \big[\frac{s}{s_{n+1}},\frac{s+1}{s_{n+1}}\big)
\end{align}  
And $u_{j,s}$ be the sequence of $\mathcal{P}$-names for 
\begin{align}
\{h_{n+1}(R_{j,t,s})\}_{t=0}^{q-1}
\end{align}

The word $u_{j,s}$ is the sequence of $\mathcal{P}$-names for the levels of the tower $\{(R^{\a_n})^t\circ h_{n+1}(R_{i,s}^{n+1})\}_{t=0}^{q}$, fr any $J_i\subset[j/(kq),(j+1)/(kq))$.

Using transects, we now describe the $\mathcal{P}$-name of the orbit of $F_0$ under $h_{n+1}\circ R^{\a_{n+1}}\circ h_{n+1}^{-1}$. We use $t$ to denote the iterations of $h_{n+1}\circ R^{\a_{n+1}}\circ h_{n+1}^{-1}$ and hence it also denotes the level of the tower.
\begin{enumerate}
\item[0.] 0-th interval $t\in [0,klq)$: 
\begin{itemize}
\item $t\in[0,lq)$: The $(h_{n+1}\circ R^{\a_{n+1}}\circ h_{n+1}^{-1},\mathcal{P})$ name agrees with the $R^{\a_n}$-name $u_{0,s}$ repeated $l$-times.
\item $t\in[lq,2lq)$: The interval crosses the boundary for the $\mathcal{I}_{kq}$ partition and the name changes to $u_{1,s}$ repeated $l$-times.
\item[.] $\ldots\ldots$
\item[]  $k-3$ more times to a total of $k$ times.
\item[.] $\ldots\ldots$
\item $t\in [(k-1)lq,klq)$: The interval crosses the boundary for the $\mathcal{I}_{kq}$ partition and the name changes to $u_{k-1,s}$ repeated $l$-times. $J_{i_{klq-1}}$ is geometrically the last subinterval of $I_1$. 
\end{itemize}

\item[1.] 1-st interval $t\in[klq, 2klq)$:
\begin{itemize}
\item $t\in [klq,(k+1)lq)$: 
\begin{itemize}
\item $t\in [klq, klq+q-j_1)$: $J_{i_{klq}}$ is geometrically the first subinterval of $I_2$. This portion of the  transect is labeled with $q-j_1$ copies of $b$. 
\item $t\in [klq+q-j_1, (k+1)lq-j_1)$: $J_{i_{klq+q-j_1}}$ is a subinterval of the geometrically first subinterval of $\mathcal{I}_{kq}$. So $F_t$ is the first letter of $u_{0,s}$. For the next $q(l-1)$ iterations the  $h_{n+1}\circ R^{\a_{n+1}}\circ h_{n+1}^{-1}$ names are the same as the $R^{\a_n}$ names and obtain the name $u_{0,s}$. 
\item $t\in [(k+1)lq-j_1, (k+1)lq)$: This segment of the transects is labeled by $j_1$ copies of $e$. 
\end{itemize}
Resulting name is $b^{q-j_1}u_{0,s}^{l-1}e^{j_1}$.

\item $t\in [(k+1)lq,(k+2)lq)$: With an identical argument we get the name $b^{q-j_1}u_{1,s}^{l-1}e^{j_1}$.

\item[.] $\ldots\ldots$
\item[] $k-3$ more times to a total of $k$ times.
\item[.] $\ldots\ldots$
\item $t\in [(k+k-1)lq,2klq)$: With an identical argument we get the name $b^{q-j_1}u_{k-1,s}^{l-1}e^{j_1}$. $J_{i_{2klq-1}}$ is the geometrically last subinterval of $I_1$.
\end{itemize}
So the name of the 1-st interval is:
\begin{align}
b^{q-j_1}u_{0,s}^{l-1}e^{j_1} b^{q-j_1}u_{1,s}^{l-1}e^{j_1} \ldots b^{q-j_1}u_{k-1,s}^{l-1}e^{j_1}
\end{align}

\item[2.] $\ldots$
\item[3.] $\ldots$
\item[] $\ldots$
\item[] $\ldots$

\item[m.] m-th interval $t\in [mklq,(m+1)klq)$: An identical argument yields the name
\begin{align}
b^{q-j_m}u_{0,s}^{l-1}e^{j_m} b^{q-j_m}u_{1,s}^{l-1}e^{j_m} \ldots b^{q-j_m}u_{k-1,s}^{l-1}e^{j_m}
\end{align}

\item[m+1.] $\ldots$
\item[] $\ldots$
\item[] $\ldots$

\item[q-1.] (q-1)-th interval $t\in [(q-1)klq,klq^2)$: An identical argument yields the name
\begin{align}
b^{q-j_{q-1}}u_{0,s}^{l-1}e^{j_{q-1}} b^{q-j_{q-1}}u_{1,s}^{l-1}e^{j_{q-1}} \ldots b^{q-j_{q-1}}u_{k-1,s}^{l-1}e^{j_{q-1}}
\end{align}
\end{enumerate}

\begin{theorem}\label{theorem 46}
Let $F_0\coloneqq h_{n+1}(R_{0,s})$ be the base of a tower $\mathcal{T}$ for $h_{n+1}\circ R^{\a_{n+1}}\circ h_{n+1}^{-1}$. Then the $\mathcal{P}$-names of $\mathcal{T}$ agree with
\begin{align}
u\coloneqq \prod_{i=0}^{q-1}\prod_{j=0}^{k-1}(b^{q-j_i}u_{j,s}^{l-1}e^{j_i})
\end{align}
on the interior of $u$. 
\end{theorem}

\begin{corollary}
With $J_i\subseteq [j/(kq),(j+1)/kq-(1/q_{n+1}))$, $j<kq$ and $R=R^n_{i,s}$, the levels $\{(R^{\a_n})^t(R)\}_{t=0}^{q-1}$ coincide with the levels $\{(h_{n+1}\circ R^{\a_{n+1}}\circ h_{n+1}^{-1})^th_{n+1}(R)\}_{t=0}^{q-1}$ in the tower for $\tau_n$. In particular their $\mathcal{Q}^*$ names agree.
\end{corollary}

\begin{corollary}
For a set $x\in X$ having measure at least $1-3/l_n$, the $(\tau_n,\mathcal{Q}^*)$ and $(\tau_n,\mathcal{Q}^*)$ names of $x$ agree on the interval $[-q,q]$.
\end{corollary}

\begin{corollary}\label{corollary 49}
Suppose that $x\in \Gamma_n$ and $x$ is on level $t_n$ of a $\tau_n$-tower and the level $t_{n+1}$ of a $\tau_{n+1}$-tower. Let $w_n$ be the $\mathcal{Q}^*$-name of $x$ with respect to $\tau_n$ and $w_{n+1}$ be the $\mathcal{Q}^*$-name of $x$ with respect to $\tau_{n+1}$. Then $w_{n+1}|_{[t_{n+1}-t_n,t_{n+1}+q_n-t_n)}=w_n$.
\end{corollary}

\subsubsection*{The symbolic representation}

We define the following two sets:
\begin{align}
B\coloneqq \{x\in\T^2: \text{ for some } m\leq n,\; x\in\Gamma_n\text{ and } H_m^{-1}(x)\in B_m\}\\
E\coloneqq \{x\in\T^2: \text{ for some } m\leq n,\; x\in\Gamma_n\text{ and } H_m^{-1}(x)\in E_m\}
\end{align}
We define the partition
\begin{align}
\mathcal{Q}\coloneqq \{A_i:i<s_0\}\cup\{B,E\}
\end{align}
where $\{A_i\}_{i=0}^{s_0}$ is the partition $\zeta_0|_{\T^2\setminus B\cup E}$. Explicitly $A_i\coloneqq H_0([0,1)\times[s/s_0,(s+1)/s_0)\setminus B\cup E)$.

We will construct $(T,\mathcal{Q})$-names for each $x\in\T^2$ using the alphabet $\Sigma\cup\{b,e\}$ where $\Sigma\coloneqq\{a_i\}_{i=0}^{s_0-1}$. So the the name of point $x\in\T^2$ will be an $f\in (\Sigma\cup\{b,e\})^\Z$ with $f(n)=a_i\iff T^n(x)\in A_i$, $f(n)=b\iff T^n(x)\in B$ and $f(n)=e\iff T^n(x)\in E$.

We use induction for a complete description of the $(T,\mathcal{Q})$-names of points in $\cup\Gamma_n$. Let $F_0\coloneqq H_{n+1}(R_{0,s^*}^{n+1})$ for some $s^*<q_{n+1}$ be the base of a tower $\mathcal{T}$ of $\tau_{n+1}$. 

Inductively we assume that the $(\tau_n,\mathcal{Q})$-names of the towers with bases $\{H_n(R_{0,s}^n)\}_{s=0}^{s_n}$ are $u_0,\ldots, u_{s_n-1}$.

At the $n+1$-th stage of the induction we start by defining words $w_0,\ldots w_{k_n-1}$ by setting
\begin{align}
w_j=u_s\iff h_{n+1}\Big(\big[\frac{j}{k_nq_n},\frac{j+1}{k_nq_n}\big)\times \big[\frac{s^*}{s_{n+1}},\frac{s^*+1}{s_{n+1}}\big)\Big)\subseteq R_{0,s}^n
\end{align}
We say that  $(w_0,\ldots,w_{k_n-1})$ is the sequence of $n$-words associated with $\mathcal{T}$.

\begin{definition}
We define a circular system by inductively specifying the sequence $\{\mathcal{W}_n\}_{n\in\N}$. We put $\mathcal{W}_0\coloneqq\{a_i\}_{i=0}^{s_0}$. After having defined $\mathcal{W}_n$ we define
\begin{align}
\mathcal{W}_{n+1}\coloneqq\{\mathcal{C}_{n+1}(w_0,\ldots,w_{k_n-1}):(w_0,\ldots,w_{k_n-1})\text{ is associated with a tower }\mathcal{T}\text{ in }\tau_{n+1}\}
\end{align}
We say that $\{\mathcal{W}_n\}_{n\in\N}$ is the construction sequence associated with the AbC construction.
\end{definition}

\begin{theorem}
Suppose $T$ is an AbC transformation satisfying Requirements 1, 2, 3 and the $l_n$ parameters are assumed to grow fast enough. 
Let $\{\mathcal{W}_n\}_{n\in\N}$ be the construction sequence associated with $T$ and let $\mathbb{K}$ be its circular system.
 
The almost all $x\in\T^2$ have $\mathcal{Q}$-names in $\mathbb{K}$. In particular there is a measure $\nu$ on $\mathbb{K}$ that makes $(\mathbb{K},\mathcal{B},\nu,\text{sh})$ isomorphic to the factor of $\T^2$ generated by $\mathcal{Q}$.
\end{theorem}

\begin{proof}
Let $M$ be any natural number. Then for a.e. $x\in\T^2$, there exists an $N=N(x)$ such that for all $n>N$ the first or last $M$ levels of $\tau_n$ does not contain $x$.

Let $x\in\Gamma_{n_0}\subset\cup_n\Gamma_n$ be an arbitrary point. Then $x$ belongs to a tower $\mathcal{T}$ of $\tau_n$ and we may in light of the previous paragraph assume that $x$ does not belong to the first of last $M$ levels of $\mathcal{T}$. Then we can use corollary \ref{corollary 49} to conclude that the if $x$ is in the $t_n$-th level of a $\tau_n$-tower, then the $T$-name of $x$ agrees with the $\tau_n$ name of $x$ on the interval $[-t_n,q_n-t_n)$.

Thus we can apply theorem \ref{theorem 46} to conclude that with $\mathcal{P}=(H_n)^{-1}(\mathcal{Q})$, we see that a tower $\mathcal{T}$ for $\tau_n$ has the name:
\begin{align}
w = \prod_{i=0}^{q_{n-1}-1}\prod_{j=0}^{k_{n-1}-1}(b^{q_n-j_i}u_{j}^{l-1}e^{j_i})
\end{align}
where $\{w_j\}_{j=0}^{k_{n-1}-1}$ is the sequence of words associated with $\mathcal{T}$. For $x$ in the $t_n$-th level of $\mathcal{T}$, the $\mathcal{Q}$-name of $x$ on the interval $[-t_n,q_n-t_n)$ is 
\begin{align}
\prod_{i=0}^{q_{n}-1}\prod_{j=0}^{k_{n}-1}(b^{q_n-j_i}u_{j}^{l-1}e^{j_i})=\mathcal{C}_n(w_0,\ldots,w_{k_{n-1}-1})
\end{align}
Since $M<\min(t_n,q_n-t_n)$ we have $x|_{[-M,M]}$ is a subword of $\mathcal{W}_n$.

It follows that the factor of $(\T^2,\mathcal{B},\mu,T)$ corresponding to the partition $\mathcal{Q}$ is a factor of the uniform circular system defined by the constructions sequence $\{\mathcal{W}_n\}_{n\in\N}$. 

Conversely, by lemma \ref{lemma 11}, the uniform circular system $\mathbb{K}$ determined by $\{\mathcal{W}_n\}_{n\in\N}$ is characterized as the smallest shift invariant closed set intersecting every basic open set $C_0(w)$ in $(\Sigma\cup\{b,e\})^\Z$ determined by some $w\in\mathcal{W}_n$. However each $w\in\mathcal{W}_n$ is represented on $\Gamma_n\cap\mathcal{T}$ for some $\mathcal{T}$, hence each $C_0(w)$ has non empty intersection with the set of words arising from $(T,\mathcal{Q})$-names. 
\end{proof}

\begin{lemma}\label{lemma 52}
Let $a_0,b_0\in\N$. Then for a.e. $x$ and large enough $n$, $x\in\Gamma_n$, $x$ does not occur in the first $a_0$ levels or last $b_0$ levels f a tower of $\tau_n$. In particular for a.e. $x\in X$ there are $a>a_0$, $b>b_0$ such that the $\mathcal{Q}$-name of $x$ restricted to the interval $[-a,b)$ belongs to $\mathcal{W}_n$.
\end{lemma}

\begin{corollary}\label{corollary 53}
For a.e. $x \in \T^2$ the $\mathcal{Q}$-name of $x$ is in $S$. In particular
if $\nu$ is the unique non-atomic shift-invariant measure on $\mathbb{K}$, then the factor of $(\T^2, \mathcal{B}, \mu, T )$ generated by $\mathcal{Q}$ is isomorphic to $(\mathbb{K}, \mathcal{B},\text{sh}, \nu)$.
\end{corollary}

\begin{lemma}\label{lemmma 54}
Suppose that for all $n$, the map sending a tower $\mathcal{T}$ for $\tau_n$ to the sequence of $\mathcal{Q}$-names associated to $\mathcal{T}$ is a one-one function. Then $\mathcal{Q}$ generates the transformation $T$.
\end{lemma}

\begin{proof}
Without loss of generality assume that $H_0$ is the identity map. Since $\{H_n\xi_n\}_{n\in\N}$ is a decreasing and generating sequence of partitions, and $\mu(\Gamma_n)\nearrow 1$, we conclude that $\{H_n\xi_n\cap\Gamma_n\}_{n=\N}$ also generate $\mathcal{B}$. So we will be done once we show that $H_n\xi_n\cap\Gamma_n$ belongs to the smallest translation invariant $\sigma$-algebra generated by $\{\vee_{n=-N}^{N} T^i(\mathcal{Q}\cup\{B,E\})\}_{N\in\N}$.

Each $P\in H_n\xi_n$ is the $t$-th level of some tower $\mathcal{T}$ for $\tau_n$. Let $w\in(\Sigma\cup\{b,e\})^{q_n}$ be the $(\tau_n,\mathcal{Q})$-name of $\mathcal{T}$. So the $j$-th letter of $w$ determines and $S_{j_i}\in(\{B\}\cup\{E\}\cup\{A_i:i<s_0\})$. Since $(\tau_n,\mathcal{Q})$-name of $\mathcal{T}$ is correct on $\Gamma_n$,
\begin{align}
P\cap\Gamma_n\subseteq \bigcap_{j=0}^{q_n-1}T^{t-j}(S_{i_j})\cap\Gamma_n
\end{align}
On the other hand, since the map sending towers to names $w$ is one to one, we see
\begin{align}
P\cap\Gamma_n\supseteq \bigcap_{j=0}^{q_n-1}T^{t-j}(S_{i_j})\cap\Gamma_n
\end{align}
This completes the proof.
\end{proof}

\begin{lemma}\label{lemma 55}
If the AbC construction is done satisfying the Requirements 1, 2 and 3 then $\mathcal{Q}$ generates the transformation $T$.
\end{lemma}

\begin{proof}
We will show that the $\mathcal{Q}$-names associated with any two $\tau_n$-towers $\mathcal{T}$ and $\mathcal{T}$ are different. 

For $n=0$ it is clear by definition. We assume that this is true upto $n$ and we show it for $n+1$. Let $R_{0,s}^{n+1}$ and $R_{0,s'}^{n+1}$ be the bases of $\mathcal{T}$ and $\mathcal{T}'$ respectively. Requirement 3 implies that the $k_n$-tuples $(j_0,\ldots,j_{k_n-1})$  and $(j_0',\ldots,j_{k_n-1}')$  associated with $s$ and $s'$ are distinct. Let $w_t$ be the names of the towers of $\tau_n$ with bases $R^n_{0,s}$. Then all the $w_t$ s are distinct as per our induction hypothesis. Also by the Theorem \ref{theorem 46}, the $\mathcal{Q}$-names of $\mathcal{T}$ and $\mathcal{T}'$ are given by $\mathcal{C}(w_{j_0},\ldots,w_{j_{k_n-1}})$ and $\mathcal{C}(w_{j_0'},\ldots,w_{j_{k_n-1}'})$ respectively. Since $(j_0,\ldots,j_{k_n-1})$  and $(j_0',\ldots,j_{k_n-1}')$ are different, $\mathcal{C}(w_{j_0},\ldots,w_{j_{k_n-1}})$ and $\mathcal{C}(w_{j_0'},\ldots,w_{j_{k_n-1}'})$ are different.

We conclude the proof using lemma \ref{lemma 55}. 
\end{proof}

\begin{theorem}
Suppose the measure preserving system $(\T^2,\mathcal{B},\mu,T)$ is built by the AbC method using fast growing coefficients and $h_n$ s in the scheme satisfies requirements  1 to 3. Let $\mathcal{Q}$ be the partition defined earlier. Then the $\mathcal{Q}$-names describe a strongly uniform circular construction sequence $\{\mathcal{W}_n\}_{n\in\N}$. Let $\mathbb{K}$ be the associated circular system and $\phi:\T^2\to\mathbb{K}$ be the map sending each $x\in \T^2$ to its $\mathcal{Q}$-name. Then $\phi$ is one to one on a set of $\mu$ measure one. Moreover, there is a unique non-atomic shift invariant measure $\nu$ concentrating on the range of $\phi$ and this measure is ergodic. In particular $(\T^2,\mathcal{B},\mu, T)$ is isomorphic to $(\mathbb{K},\mathcal{B},\nu,\text{sh})$. 
\end{theorem}

\begin{proof}
With fast growing $l_n$ parameters, we know that the sequence $\mathcal{W}_n$ forms a uniform circular sequence and hence there is a unique shift invariant non-atomic ergodic measure $\nu$ on $S\subset\mathbb{K}$. Using lemma \ref{lemma 52}, we have that the range of $\phi$ is a subset of the set $S$. So the factor determined by $\phi$ is isomorphic to $(S,\mathcal{B},\mu,\text{sh})$

Since requirements 1,2 and 3 are satisfied, the partition $\mathcal{Q}$ generates $\T^2$ and hence $\phi$ is an isomorphism.
\end{proof}

\begin{maintheorem}\label{theorem 57}
Consider three sequence of natural numbers $\{k_n\}_{n\in\N},\{l_n\}_{n\in\N},\{s_n\}_{n\in\N}$ tending to infinity. Assume that 
\begin{enumerate}
\item $l_n$ grows sufficiently fast.
\item $s_n$ divides both $k_n$ and $s_{n+1}$.
\end{enumerate} 
Let $\{\{\mathcal{W}_n\}\}_{n\in\N}$ be a circular construction sequence in an alphabet $\Sigma\cup\{b,e\}$ such that
\begin{enumerate}
\item $\mathcal{W}_0=\Sigma$, $|\mathcal{W}_{n+1}|=s_{n+1}$ for any $n\geq 1$.
\item (Uniform) For each $w'\in\mathcal{W}_{n+1}$, and $w\in\mathcal{W}_n$, if $w'=\mathcal{C}(w_0,\ldots,w_{K_n-1})$, then here are $k_n/s_n$ many $j$ with $w = w_j$.
\end{enumerate}
Then,
\begin{enumerate}
\item $\{\mathcal{W}_n\}_{n\in\N}$ is a uniform construction sequence. If $\mathbb{K}$ is the associated symbolic shift then there is a unique non atomic ergodic measure $\nu$ on $\mathbb{K}$.
\item There is a real-analytic measure preserving transformation $T$ defined on $\T^2$ such that the system $(\T^2,\mathcal{B},\mu,T)$ is isomorphic to $(\mathbb{K},\mathcal{B},\nu,sh)$.
\end{enumerate}
\end{maintheorem}

\begin{proof}
Suppose that we have defined $\{h_m\}_{m=0}^{n}$ in the AbC process. From the definition of uniform circular systems (see \ref{definition circular systems}), we can find $P_{n+1}\subseteq(W_n)^{k_n}$ such that $\mathcal{W}_{n+1}$ is the collection of $w'$ such that for some sequence $(w_0,\ldots,w_{k_n-1})\in P_{n+1},w'=\mathcal{C}(w_0,\ldots, w_{k_n-1})$. We enumerate $P_{n+1}=\{w_0^{\prime},\ldots,w_{s_{n+1}-1}^{\prime}\}$. We apply lemma \ref{lemma generating 1} to get $h_{n+1}$ from $w_0^{\prime},\ldots,w_{s_{n+1}-1}^{\prime}$.

With $\{h_n\}_{n\in\N}$ constructed as above, we note that requirement 1 is satisfied because $s_n\nearrow \infty$, requirement 2 and 3 follows from the fact that the words in $P_{n+1}$ are distinct. 
\end{proof}

We note that the sequence $\{P_n\}_{n\in\N}$ determines $h_n$ in the AbC method which in turn determines a  neighborhood in the real-analytic topology in which the resulting AbC diffeomorphism belongs. Conversely, different choices of $P_n$ gives distant $h_n$'s and hence distant $h_n^{(\mathfrak{a})}$ in the analytic topology. We record this as this is going to be crucial in the proof of an anti-classification result. 

\begin{proposition}\label{proposition 58}
Suppose $\{\mathcal{U}_n\}_n\in\N$ and $\{\mathcal{W}_n\}_n\in\N$ are two construction sequences for circular systems and $M$ is such that $\mathcal{U}_n=\mathcal{W}_n\;\forall n\leq M$. If $S$ and $T$ are the real-analytic realizations of the circular systems using the AbC method given in this paper, then
\begin{align}
d_\rho(S,T)<\e_M
\end{align}  
\end{proposition}

\begin{proof}
The sequences $\{k_n^\mathcal{U},l_n^\mathcal{U},h_n^\mathcal{U},s_n^\mathcal{U}\}_{n\in\N}$ and $\{k_n^\mathcal{W},l_n^\mathcal{W},h_n^\mathcal{W},s_n^\mathcal{W}\}_{n\in\N}$ associated with the two construction sequences determining approximations $\{S_n\}_{n\in\N}$ and $\{T_n\}_{n\in\N}$ to diffeomorphisms $S$ and $T$ has the property 
 $k_n^\mathcal{U}=k_n^\mathcal{W},l_n^\mathcal{U}=l_n^\mathcal{W},h_n^\mathcal{U}=h_n^\mathcal{W},s_n^\mathcal{U}=s_n^\mathcal{W}$ for all $n\leq M$. So $S_M=T_M$. It follows from equations \ref{eq:epsdecrease} and \ref{equation 26} that
 \begin{align}
 d_\rho(S_m,S)<\e_M/2\\
 d_\rho(T_m,T)<\e_M/2
 \end{align}
Combining the two together with the triangle inequality we get the result.
\end{proof}

\section{Proof of Theorem \ref{theo:main anti}} \label{section main proof}
In this Section we survey the key steps from \cite{FRW} and \cite{FW3} and adapt them to our proof of Theorem \ref{theo:main anti} in the real-analytic category.

\subsection{Trees, groups and equivalence relations} \label{subsec:trees}
During the construction the following maps will prove useful.
\begin{definition}
\label{def:M-and-s}We define a map $M:\mathcal{T}rees\to\mathbb{N}^{\mathbb{N}}$
by setting $M\left(\mathcal{T}\right)(s)=n$ if and only if $n$ is
the least number such that $\sigma_{n}\in\mathcal{T}$ and $\abs{\sigma_{n}}=s$.
Dually, we also define a map $s:\mathcal{T}rees\to\mathbb{N}^{\mathbb{N}}$
by setting $s\left(\mathcal{T}\right)(n)$ to be the length of the
longest sequence $\sigma_{m}\in\mathcal{T}$ with $m\leq n$. 
\end{definition}

\begin{remark}
When $\mathcal{T}$ is clear from the context we write $M(s)$ and
$s(n)$. We also note that $s(n)\leq n$ and that $s$ as well as
$M$ are continuous functions when we endow $\mathbb{N}$ with the
discrete topology and $\mathbb{N}^{\mathbb{N}}$ with the product
topology.
\end{remark}

Related to the structure of the tree we will specify several equivalence
relations $\mathcal{Q}_{s}^{n}$ on the words $\mathcal{W}_{n}$. For this purpose, we recall
some general notions on equivalence relations.
\begin{definition}
\label{def:equivrel}Let $X$ be a set and $\mathcal{Q}$ as well
as $\mathcal{R}$ equivalence relations on $X$.
\begin{itemize}
\item We write $\mathcal{R}\subseteq\mathcal{Q}$ and say that $\mathcal{R}$
\emph{refines} $\mathcal{Q}$ if considered as sets of ordered pairs
we have $\mathcal{R}\subseteq\mathcal{Q}$.
\item We define the \emph{product equivalence relation} $\mathcal{Q}^{n}$
on $X^{n}$ by setting $x_{0}x_{1}\dots x_{n-1}\sim x_{0}^{\prime}x_{1}^{\prime}\dots x_{n-1}^{\prime}$
if and only if $x_{i}\sim x_{i}^{\prime}$ for all $i=0,1,\dots,n-1$.
\end{itemize}
\end{definition}

In \cite{FRW} groups associated to trees are used to approximate conjugacies. These groups are direct sums of $\mathbb{Z}_{2}=\mathbb{Z}/2\mathbb{Z}$
and are called \emph{groups of involutions}. 

If $G=\sum_{i\in I}\left(\mathbb{Z}_{2}\right)_{i}$ and $B=\left\{ r_{i}:i\in I\right\} $
is a distinguished basis, then we call the elements $r_{i}\in B$
\emph{generators} and we have a well-defined notion of \emph{parity}
for elements in $G$: an element $g\in G$ is called \emph{even} if
it can be written as the sum of an even number of elements in $B$.
Otherwise, it is called \emph{odd}. Parity is preserved under homomorphisms
sending the distinguished basis of one group to the distinguished
basis of the other. This also yields that for an inverse limit system
of groups of involutions $\left\{ G_{s}:s\in\mathbb{N}\right\} $
(where each group has a distinguished set of generators) with homomorphisms
$\rho_{t,s}:G_{t}\to G_{s}$ for $s<t$, the elements of the inverse
limit $\underleftarrow{\lim}G_{s}$ have a well defined parity. 

For a tree $\mathcal{T}\subset\mathbb{N}^{\mathbb{N}}$ we assign
to each level $s$ a group of involutions $G_{s}\left(\mathcal{T}\right)$
by taking a sum of copies of $\mathbb{Z}_{2}$ indexed by the nodes
of $\mathcal{T}$ at level $s$. Moreover, we view these nodes of
$\mathcal{T}$ at level $s$ as the distinguished generators of 
\[
G_{s}\left(\mathcal{T}\right)=\sum_{\tau\in\mathcal{T},\,\abs{\tau}=s}\left(\mathbb{Z}_{2}\right)_{\tau}.
\]
For levels $s<t$ of $\mathcal{T}$ we have a canonical homomorphism
$\rho_{t,s}:G_{t}\left(\mathcal{T}\right)\to G_{s}\left(\mathcal{T}\right)$
that sends a generator $\tau$ of $G_{t}\left(\mathcal{T}\right)$
to the unique generator $\sigma$ of $G_{s}\left(\mathcal{T}\right)$
that is an initial segment of $\tau$. We denote the inverse limit
of $\left\langle G_{s}\left(\mathcal{T}\right),\rho_{t,s}:s<t\right\rangle $
by $G_{\infty}\left(\mathcal{T}\right)$ and we let $\rho_{s}:G_{\infty}\left(\mathcal{T}\right)\to G_{s}\left(\mathcal{T}\right)$
be the projection map.

Since there is a one-to-one correspondence between the infinite branches
of $\mathcal{T}$ and infinite sequences $\left(g_{s}\right)_{s\in\mathbb{N}}$
of generators $g_{s}\in G_{s}\left(\mathcal{T}\right)$ with $\rho_{t,s}\left(g_{t}\right)=g_{s}$
for $t>s$, we obtain the following characterization.
\begin{fact} \label{fact:illgroups}
Let $\mathcal{T}\subset\mathbb{N}^{\mathbb{N}}$ be a tree. Then 
\begin{enumerate}
    \item $G_{\infty}\left(\mathcal{T}\right)$
has a nonidentity element of odd parity if and only if $\mathcal{T}$
is ill-founded (i.e. has an infinite branch).
\item $G_{\infty}\left(\mathcal{T}\right)$
has a nonidentity element of even parity if and only if $\mathcal{T}$
has at least two infinite branches).
\end{enumerate} 
\end{fact}

In order to make the elements of $G_{\infty}\left(\mathcal{T}\right)$
to correspond to conjugacies, one builds symmetries into our construction
using the following finite approximations to $G_{\infty}\left(\mathcal{T}\right)$.
We let $G_{0}^{n}\left(\mathcal{T}\right)$ be the trivial group and
for $s>0$ we let 
\[
G_{s}^{n}\left(\mathcal{T}\right)=\sum\left(\mathbb{Z}_{2}\right)_{\tau},\text{ where the sum is taken over }\tau\in\mathcal{T}\cap\left\{ \sigma_{m}:m\leq n\right\} ,\,\abs{\tau}=s.
\]

When $\mathcal{T}$ is clear from the context, we will frequently
write $G_{s}^{n}$.

During the course of construction we will consider group actions of
$G_{s}^{n}$ on our quotient spaces $\mathcal{W}_{n}/\mathcal{Q}_{s}^{n}$.
Here, we will need to control systems of such group actions on the
refining equivalence relations. For that purpose, the following general
definitions will prove useful.
\begin{definition}
Suppose
\begin{itemize}
\item $\mathcal{Q}$ and $\mathcal{R}$ are equivalence relations on a set
$X$ with $\mathcal{R}$ refining $\mathcal{Q}$,
\item $G$ and $H$ are groups with $G$ acting on $X/\mathcal{Q}$ and
$H$ acting on $X/\mathcal{R}$,
\item $\rho:H\to G$ is a homomorphism.
\end{itemize}
Then we say that the $H$ action is \emph{subordinate} to the $G$
action if for all $x\in X$, whenever $[x]_{\mathcal{R}}\subset[x]_{\mathcal{Q}}$
we have $h[x]_{\mathcal{R}}\subset\rho(h)[x]_{\mathcal{Q}}$. 
\end{definition}

\begin{definition}
If $G$ acts on $X$, then the canonical \emph{diagonal action} of
$G$ on $X^{n}$ is defined by 
\[
g\left(x_{0}x_{1}\dots x_{n-1}\right)=gx_{0}\,gx_{1}\dots gx_{n-1}.
\]
If $G$ is a group of involutions with a distinguished collection
of free generators, then we define the \emph{skew diagonal action}
on $X^{n}$ by setting 
\[
g\left(x_{0}x_{1}\dots x_{n-1}\right)=gx_{n-1}\,gx_{n-2}\dots gx_{0}
\]
for any canonical generator $g$.
\end{definition}

\begin{remark}
We stress that the skew diagonal actions by elements of $G$ of odd
parity reverse the orders of the $x_{i}$'s while group elements of
even parity preserve the order. 
\end{remark}
\begin{remark}
Recalling the notion of a product equivalence relation from Definition
\ref{def:equivrel} we can identify $X^{n}/\mathcal{Q}^{n}$ with
$\left(X/\mathcal{Q}\right)^{n}$ in an obvious way. Then we can also
extend an action of $G$ on $X/\mathcal{Q}$ to the diagonal or skew
diagonal actions on $\left(X/\mathcal{Q}\right)^{n}$ in a straightforward
way.
\end{remark}

\subsection{Specifications and timing assumptions} \label{subsec:speci}

In \cite{FRW} the sequences $(\mathtt{W}_n)_{n\in \N}$, equivalence relations $\mathcal{Q}^n_s$, and group actions of $G^n_s$ satisfying several specifications are constructed inductively. For each $n\in \N$ with $\sigma_n \in \TT$ one constructs a set of words $\mathtt{W}_n=\mathtt{W}_n(\TT)$ in the basic alphabet $\{0,1\}$ and the construction depends only on $\TT\cap \Meng{\sigma_m}{m\leq n}$. To start we set $\mathtt{W}_0=\{0,1\}$. 

\begin{itemize}
    \item[(E1)] All words in $\mathtt{W}_n$ have the same length $\mathtt{h}_n$ and $|\mathtt{W}_n|\coloneqq s_n$ is a power of $2$.
    \item[(E2)] If $\sigma_m$ and $\sigma_n$ are consecutive elements of $\TT$, then every word in $\mathtt{W}_{n}$ is built by concatenating words in $\mathtt{W}_m$. Every word in $\mathtt{W}_m$ occurs in each word of $\mathtt{W}_{n}$ exactly $p^2_n$ many times. Here, $p_n$ is a large prime number chosen when $\sigma_n$ is considered.
    \item[(E3)] If $\sigma_m$ and $\sigma_n$ are consecutive elements of $\TT$, $w\in \mathtt{W}_{n}$ and $w=pw_1\dots w_l s$, where each $w_i \in \mathtt{W}_m$ and $p$ and $s$ are strings in $\{0,1\}$ of length less than the length of the $m$-words, then both $p$ and $s$ are the empty words.
    
    If $w,w'\in \mathtt{W}_{n}$, then the first half of $w'$ is not equal to $w$, i.e. if $w=w_1\dots w_{\mathtt{h}_n/\mathtt{h}_m}$ and $w'=w^{\prime}_1 \dots w^{\prime}_{\mathtt{h}_n/\mathtt{h}_m}$, where $w_i,w^{\prime}_i \in \mathtt{W}_m$, and $k=\lfloor \frac{\mathtt{h}_n/\mathtt{h}_m}{2} \rfloor +1$, we have $w_k \dots w_{\mathtt{h}_n/\mathtt{h}_m} \neq w^{\prime}_1 \dots w^{\prime}_{\mathtt{h}_n/\mathtt{h}_m -k-1}$.
\end{itemize}

\begin{remark}
In particular, these specifications say that $\Meng{\mathtt{W}_n}{\sigma_n \in \TT}$ is a uniquely readable and strongly uniform construction sequence.
\end{remark}

In the next step, we specify equivalence relations $\mathcal{Q}^n_s$ on $\mathtt{W}_n$ which are defined for $s\leq s(n)$. For this purpose, we take a decreasing summable sequence $\left(\epsilon_{n}\right)_{n\in\mathbb{N}}$ and a strictly increasing sequence $e(n)$ of positive integers. To start, we let $\mathcal{Q}_{0}^{0}$ be the equivalence relation on $\mathtt{W}_{0}=\{0,1\}$ which has one equivalence class, i.e. both elements of $\{0,1\}$ are equivalent. 
\begin{itemize}
\item[(Q4)] Suppose that $n=M(s)$. Then any two words in the
same $\mathcal{Q}_{s}^{n}$ class agree on an initial segment of length at least $\left(1-\epsilon_{n}\right)\mathtt{h}_{n}$.
\item[(Q5)] For $n\geq M(s)+1$ we can consider words in $\mathtt{W}_{n}$
as concatenation of words from $\mathtt{W}_{M(s)}$ and define $\mathcal{Q}_{s}^{n}$
as the product equivalence relation of $Q_{s}^{M(s)}$.
\item[(Q6)] $\mathcal{Q}_{s+1}^{n}$ refines $\mathcal{Q}_{s}^{n}$ and each
$\mathcal{Q}_{s}^{n}$ class contains $2^{e(n)}$ many $\mathcal{Q}_{s+1}^{n}$
classes.
\end{itemize}

\begin{remark}
Each equivalence relation $\mathcal{Q}_{s}^{n}$ will induce an equivalence
relation on $rev(\mathtt{W}_{n})$, which we will also call $\mathcal{Q}_{s}^{n}$,
as follows: $rev(w),rev(w')\in rev(\mathtt{W}_{n})$ are equivalent
with respect to $\mathcal{Q}_{s}^{n}$ if and only if $w,w'\in\mathtt{W}_{n}$
are equivalent with respect to $\mathcal{Q}_{s}^{n}$. 
\end{remark}
\begin{remark}
By (Q5) we can view $\mathtt{W}_{n}/\mathcal{Q}_{s}^{n}$ as sequence
of elements $\mathtt{W}_{M(s)}/\mathcal{Q}_{s}^{M(s)}$ and similarly
for $rev(\mathtt{W}_{n})/\mathcal{Q}_{s}^{n}$. It also follows that $\mathcal{Q}_{0}^{n}$ is the equivalence relation
on $\mathtt{W}_{n}$ which has one equivalence class. 
\end{remark}
\begin{remark}
We denote the number of $\mathcal{Q}^n_s$ equivalence classes by $Q^n_s$ and the cardinality of each class by $C^n_s$. In case that the exponent is not relevant we will refer to the $\mathcal{Q}_{s}^{n}$
as $\mathcal{Q}_{s}$. For $u\in\mathtt{W}_{n}$ we write $[u]_{s}$
for its $\mathcal{Q}_{s}^{n}$ class.
\end{remark}

We now list specifications on the group actions.
\begin{itemize}
\item[(A7)] $G_{s}^{n}$ acts freely on $\mathtt{W}_{n}/\mathcal{Q}_{s}^{n}\cup rev(\mathtt{W}_{n}/\mathcal{Q}_{s}^{n})$
and the $G_{s}^{n}$ action is subordinate to the $G_{s-1}^{n}$ action
via the natural homomorphism $\rho_{s,s-1}:G_{s}^{n}\to G_{s-1}^{n}$.
\item[(A8)] The canonical generator of $G^{M(s)}_s$ sends elements of $\mathtt{W}_{M(s)}/\mathcal{Q}_{s}^{M(s)}$ to elements of $rev(\mathtt{W}_{M(s)})/\mathcal{Q}_{s}^{M(s)}$ and vice versa.
\item[(A9)] Suppose $M(s)<n$, $\sigma_{m}$ and $\sigma_{n}$ are consecutive
elements of $\mathcal{T}$ and we view $G_{s}^{n}=G_{s}^{m}\oplus H$.
Then the action of $G_{s}^{m}$ on $\mathtt{W}_{m}/\mathcal{Q}_{s}^{m}\cup rev(\mathtt{W}_{m}/\mathcal{Q}_{s}^{m})$
is extended to an action on action on $\mathtt{W}_{n}/\mathcal{Q}_{s}^{n}\cup rev(\mathtt{W}_{n}/\mathcal{Q}_{s}^{n})$
by the skew diagonal action. If $H$ is nontrivial, then its canonical generator maps $\mathtt{W}_{n}/\mathcal{Q}_{s}^{n}$ to $rev(\mathtt{W}_{n}/\mathcal{Q}_{s}^{n})$.
\end{itemize}

All these specifications build certain symmetries into the words that allow nodes of the tree to give increasingly precise information about invertible graph joinings. The intent of the following specifications is to give a mechanism for showing that any joining not arising from branches through the tree are independent joinings over a non-trivial factor, i.e. they cannot give an isomorphism.

\begin{itemize}
    \item[(J10)] Suppose $\sigma_m$ and $\sigma_n$ are consecutive elements of $\TT$. Let $u$ and $v$ be elements of $\mathtt{W}_n \cup rev(\mathtt{W}_n)$ and let $1\leq t < (1-\epsilon_n)\frac{\mathtt{h}_n}{\mathtt{h}_m}$. Let $j_0$ be a number between $\varepsilon_n\frac{\mathtt{h}_n}{\mathtt{h}_m}$ and $\frac{\mathtt{h}_n}{\mathtt{h}_m}-t$. Then for each pair $u',v' \in \mathtt{W}_m \cup rev(\mathtt{W}_m)$ such that $u'$ has the same parity as $u$ and $v'$ has the same parity as $v$, let $r(u',v')$ be the number of $j<j_0$ such that $(u',v')$ occurs in $(sh^{t\mathtt{h}_m}(u),v)$ in the $(j\mathtt{h}_m)$-th position in their overlap. Then 
    \[
    \abs{\frac{r(u',v')}{j_0}-\frac{1}{s^2_0}}<\epsilon_n.
    \]
    \item[(J11)] Suppose $\sigma_m$ and $\sigma_n$ are consecutive elements of $\TT$. Let $u \in \mathtt{W}_n$ and $v\in \mathtt{W}_n \cup rev(\mathtt{W}_n)$. We let $s=s(u,v)$ be the maximal $i$ such that there is a $g\in G^m_i$ such that $g[u]_i=[v]_i$. Let $g=g(u,v)$ be the unique $g$ with this property and $(u',v')\in \mathtt{W}_m \times \left(\mathtt{W}_m\cup rev(\mathtt{W}_m)\right)$ such that $g[u']_s=[v']_s$. Let $r(u',v')$ be the number of occurrences of $(u',v')$ in $(u,v)$. Then:
    \[
    \abs{\frac{r(u',v')}{\mathtt{h}_n/\mathtt{h}_m}-\frac{1}{Q^n_s} \left(\frac{1}{C^n_s}\right)^2}<\epsilon_n.
    \]
\end{itemize}

To satisfy so-called \emph{timing assumptions} in \cite[Section 7.2.1]{FW3} for the corresponding circular system one also needs this strengthening of a special case.
\begin{itemize}
    \item[(J11.1)] Suppose $\sigma_m$ and $\sigma_n$ are consecutive elements of $\TT$. Let $u \in \mathtt{W}_n$, $v\in \mathtt{W}_n \cup rev(\mathtt{W}_n)$, and $[u]_1\notin G^n_1[v]_1$. Let $j_0$ be a number between $\epsilon_n\frac{\mathtt{h}_n}{\mathtt{h}_m}$ and $\frac{\mathtt{h}_n}{\mathtt{h}_m}$. Suppose that $I$ is either an initial or a tail segment of the interval $[0,\mathtt{h}_n-1]\cap \Z$ having length $j_0\mathtt{h}_m$. Then for every pair $(u',v')\in \mathtt{W}_m \times \left(\mathtt{W}_m\cup rev(\mathtt{W}_m)\right)$ such that $u'$ has the same parity as $u$ and $v'$ has the same parity as $v$, let $r(u',v')$ be the number of occurrences of $(u',v')$ in $(u\upharpoonright I , v\upharpoonright I)$. Then: 
    \[
    \abs{\frac{r(u',v')}{j_0}-\frac{1}{s^2_n}}<\epsilon_n.
    \]
\end{itemize}

Under these specifications and the timing assumptions the following existence result for synchronous and anti-synchronous isomorphisms is shown in \cite[Theorem 92]{FW3}.

\begin{fact}\label{fact:FW3main}
Suppose $\K^c$ is a circular system satisfying the timing assumptions and view it as an element $T$ of $\mathbf{MPT}$. Then:
\begin{enumerate}
    \item If there is an isomorphism $\phi:\K^c \to \K^c$ such that $\phi \notin \overline{\Meng{T^n}{n\in \Z}}$, then there is an isomorphism $\psi:\K^c \to \K^c$ such that $\phi \notin \overline{\Meng{T^n}{n\in \Z}}$ and $\psi^\pi$ is the identity map.
    \item If there is an isomorphism $\phi:\K^c \to (\K^c)^{-1}$, then there is an isomorphism $\psi:\K^c \to (\K^c)^{-1}$ such that $\psi^\pi=\natural$.
\end{enumerate}
\end{fact}

\subsection{The continuous function $F: \mathcal{T}rees \to \mathbf{MPT}$} \label{subsec:FRWmap}
The specifications allow to construct the following continuous map $F: \mathcal{T}rees \to \mathbf{MPT}$, which is the main result of \cite{FRW}. 

\begin{proposition} \label{prop:FRWmain}
There is a continuous one-to-one map $F: \mathcal{T}rees \to \mathbf{MPT}$ such that for $\TT \in \Trees$, if $T=F^s(\TT)$:
\begin{enumerate}
    \item $\TT$ has an infinite branch if and only if $T \cong T^{-1}$
    \item $\TT$ has two distinct infinite branches if and only if 
    \[
    C(T) \neq \overline{ \Meng{T^n}{n\in \Z}}.
    \]
\end{enumerate}
\end{proposition}

More specifically, the following facts for the map $F$ hold true.

\begin{fact} \label{fact:antisyn}
The transformations in the range of $F$ are strongly uniform odometer based transformations and for $S$ in the range of $F$ we have $S\cong S^{-1}$ if and only if there is an anti-synchronous isomorphism between $S$ and $S^{-1}$.
\end{fact}

\begin{fact} \label{fact:syn}
If $S$ is in the range of $F$ and $C(S)\neq \Meng{S^n}{n\in \Z}$, then there is a synchronous isomorphism $\phi \in C(S)\setminus \Meng{S^n}{n\in \Z}$ such that for some $n\in \N$ there is a non-identity element $g\in G^n_1$ satisfying for all generic $s\in \mathbb{K}$ and all large enough $m\in \N$ that if $u$ and $v$ are the principal $m$-subwords of $s$ and $\phi(s)$, respectively, then $[v]_1=g[u]_1$.  
\end{fact}

\begin{fact} \label{fact:continuousTree}
For every $N\in \N$ there is $M\in \N$ such that for $\mathcal{T},\TT'\in \Trees$ with $\TT\cap \Meng{\sigma_n}{n\leq M}=\TT'\cap \Meng{\sigma_n}{n\leq M}$ the first $N$ steps of the construction sequence for $F(\TT)$ are equal to the first $N$ steps of the construction sequence for $F(\TT')$, i.e. $\left(\mathtt{W}_k(\TT)\right)_{k< N}=\left(\mathtt{W}_k(\TT')\right)_{k< N}$.
\end{fact}

\subsection{Numerical requirements} \label{subsec:numerical}
To construct the map $F:\mathcal{T}rees\to \mathbf{MPT}$ one builds
a construction sequence $\left(\mathcal{W}_{n}\left(\mathcal{T}\right)\right)_{n\in\mathbb{N}}$
satisfying our specifications for each $\mathcal{T}\in\mathcal{T}rees$.
Moreover, this has to be done in such a way that $\mathcal{W}_{n}\left(\mathcal{T}\right)$
is entirely determined by $\mathcal{T}\cap\left\{ \sigma_{m}:m\leq n\right\} $.
Therefore, in \cite{FRW} the construction is organized as follows: For
each $n$ and for each subtree $\mathcal{S}\subseteq\left\{ \sigma_{m}:m\leq n\right\} $
and each $\sigma_{m}\in\mathcal{S}$ we build $\mathcal{W}_{m}\left(\mathcal{S}\right)$.
By induction we want to pass from stage $n-1$ to stage $n$. So,
we assume that we have constructed $\left(\mathcal{W}_{m}\left(\mathcal{S}\right)\right)_{\sigma_{m}\in\mathcal{S}}$
for each subtree $\mathcal{S}\subseteq\left\{ \sigma_{m}:m\leq n-1\right\} $.
In the inductive step, we now have to construct $\left(\mathcal{W}_{m}\left(\mathcal{S}\right)\right)_{\sigma_{m}\in\mathcal{S}}$
for each subtree $\mathcal{S}\subseteq\left\{ \sigma_{m}:m\leq n\right\} $. 

Obviously, for those trees $\mathcal{S}$ with $\sigma_{n}\notin\mathcal{S}$
there is nothing to do. So, one has to work on the finitely many trees
$\mathcal{S}$ with $\sigma_{n}\in\mathcal{S}$. We list those in
arbitrary manner as $\left\{ \mathcal{S}_{1},\dots,\mathcal{S}_{E}\right\} $.
Assume that $\mathcal{T}$ is the $e$-th tree on this list and that
we have constructed the collections $\mathcal{W}_{n}\left(\mathcal{S}_{i}\right)$
of $n$-words for all $i\in\left\{ 1,\dots,e-1\right\} $. Let $\mathfrak{P}$
be the collection of prime numbers occurring in the prime factorization
of any of the lengths of the words that we have constructed so far. Then one wants to
construct $\mathcal{W}_{n}\left(\mathcal{T}\right)$, $\mathcal{Q}_{s}^{n}\left(\mathcal{T}\right)$,
and $G_{s}^{n}\left(\mathcal{T}\right)$. For that purpose, one uses the induction assumption that for  
$m$ being the largest number less than $n$ such that $\sigma_{m}\in\mathcal{T}$ 
we have $\mathcal{W}_{m}=\mathcal{W}_{m}\left(\mathcal{T}\right)$,
$\mathcal{Q}_{s}^{m}=\mathcal{Q}_{s}^{m}\left(\mathcal{T}\right)$,
and $G_{s}^{m}=G_{s}^{m}\left(\mathcal{T}\right)$ satisfying our
specifications. 

During the course of the construction several parameters with numerical conditions about growth and decay rates appear. In \cite[Section 10]{FW3} all their recursive requirements and interdependencies are listed and their consistency is resolved. In particular, at stage $n$ the parameter $l_n$ is chosen last. Hence, we can choose it sufficiently large to guarantee the convergence to a real-analytic diffeomorphism in our Theorem \ref{theorem 57}. 

\subsection{$F^s=R\circ \mathcal{F} \circ F$ is a continuous reduction} \label{subsec:reduction}
Finally, we are ready to prove Theorem \ref{theo:main anti}. Therefor, we recall that Theorem \ref{theorem 57} gives to every strongly uniform circular construction sequence $\{W_n\}_{n\in \mathbb{N}}$ with $|W_n|=s_n\to \infty$ and circular coefficients $(k_n ,l_n )_{n\in \mathbb{N}}$, where $(l_n)_{n\in\N}$ grows sufficiently fast, a diffeomorphism $T\in \diffr$ measure theoretically isomorphic to the circular system $\mathbb{K}^c$. Hereby, we obtain a map from circular systems with fast growing coefficients to $\diffr$. We denote this map with our choice of the sequence $(l_n)_{n\in \N}$ from the previous Subsection by $R$ and point out that it preserves isomorphism. 

Hereby, we define the map $F^s=R\circ \mathcal{F} \circ F: \mathcal{T}rees \to \diffr$ and show that it is a continuous reduction. 

\begin{lemma}
$F^s: \mathcal{T}rees \to \diffr$ is a continuous function.
\end{lemma}

\begin{proof}
Let $T=F^s(\mathcal{T})$ for $\mathcal{T} \in \mathcal{T}rees$ and $U$ be an open neighborhood of $T$ in $\diffr$. Let $T^c=\mathcal{F} \circ F(\mathcal{T})$ be the circular system such that $R(T^c)=T$. By Proposition \ref{proposition 58} there is $M\in\N$ sufficiently large such that for all $S\in \diffr$ in the range of $R$ we have the following property: If $\left(\mathcal{W}_n(T^c)\right)_{n\leq M}=\left(\mathcal{W}_n(S^c)\right)_{n\leq M}$, then $S\in U$. Here, $S^c$ denotes the circular system such that $S=R(S^c)$. Moreover, $\left(\mathcal{W}_n(T^c)\right)_{n\in \N}$ and $\left(\mathcal{W}_n(S^c)\right)_{n\in \N}$ denote the construction sequences of $T^c$ and $S^c$, respectively. We recall from Subsection \ref{subsec:catfunc} that $\left(\mathcal{W}_n(T^c)\right)_{n\leq M}$ is determined by the first $M+1$ members in the construction sequence of the odometer based system $F(\mathcal{T})$, i.e. it is determined by $\left(\mathtt{W}_n(\mathcal{T})\right)_{n\leq M}$. By Fact \ref{fact:continuousTree} there is a basic open set $V\subseteq \mathcal{T}rees$ containing $\mathcal{T}$ such that for all $\mathcal{S} \in \mathcal{T}rees$ the first $M+1$ members of the construction sequences of $F(\mathcal{T})$ and $F(\mathcal{S})$ are the same, i.e. $\left(\mathtt{W}_n(\mathcal{T})\right)_{n\leq M}=\left(\mathtt{W}_n(\mathcal{S})\right)_{n\leq M}$. Then it follows that $F^s(\mathcal{S})\in U$ for all $\mathcal{S} \in V$.
\end{proof}

To conclude the proof of Theorem \ref{theo:main anti} we show that $F^s$ is a reduction.
\begin{proof}[Proof of Theorem \ref{theo:main anti}]
To see that $F^s$ is a reduction, it suffices to check that $\mathcal{F} \circ F$ is a reduction because $R$ preserves isomorphism. This follows from \cite[Section 8.1]{FW3} and we only present the proof of the first part of the Theorem. Here, we let $T=F^s(\TT)$, $\K=F(\TT)$, and $\K^c=\mathcal{F}\circ F(\TT)$ for $\TT\in \Trees$.

Suppose $\TT\in \Trees$ has an infinite branch. By Proposition \ref{prop:FRWmain} and Fact \ref{fact:antisyn} there is an anti-synchronous isomorphism $\phi:\K \to \K^{-1}$. Then we can apply Fact \ref{fact:functor} on the functor $\mathcal{F}$ and obtain that there is an anti-synchronous isomorphism $\phi^c:\K^c\to (\K^c)^{-1}$.
    
To show the converse direction we suppose that $T\cong T^{-1}$, which yields $\K^c\cong (\K^c)^{-1}$. Since the construction sequence $(\mathcal{W}_n)_{n\in \N}$ for $\K^c$ satisfies the timing assumptions, Fact \ref{fact:FW3main} shows that there is an anti-synchronous isomorphism $\phi^c:\K^c\to (\K^c)^{-1}$. Hence, we can apply Fact \ref{fact:functor} again and obtain that there is an anti-synchronous isomorphism $\phi:\K \to \K^{-1}$. By Proposition \ref{prop:FRWmain} $\TT$ has an infinite branch.
\end{proof}

\section{Further applications}

\subsection{An uncountable family of pairwise non-Kakutani equivalent real-analytic diffeomorphisms} \label{subsec:uncountable}
We recall the definition of the $\overline{f}$ distance between strings of symbols introduced by Feldman \cite{Fe} (a zero entropy version of this property was introduced independently by Katok \cite{K77}).
\begin{definition}
\label{def:fbar}A \emph{match} between two strings of symbols $a_{1}a_{2}\dots a_{n}$ and $b_{1}b_{2}\dots b_{m}$
from a given alphabet $\Sigma,$ is a collection $I$ of pairs of indices $(i_s,j_s)$, $s=1,\dots,r$ such that $1\le i_1 <i_2<\cdots <i_r\le n$, $1\le j_1 <j_2 <\cdots <j_r\le m$ and $a_{i_s}=b_{i_s}$ for $s=1,2,\dots,r.$ Then 
\begin{equation*}
\overline{f}(a_{1}a_{2}\dots a_{n},b_{1}b_{2}\dots b_{m}) = \displaystyle{1-\frac{2\sup\{|I|:I\text{\ is\ a\ match\ between\ }a_{1}a_{2}\cdots a_{n}\text{\ and\ }b_{1}b_{2}\cdots b_{m}\}}{n+m}.}
\end{equation*} 
\end{definition}
We will refer to $\overline{f}(a_{1}a_{2}\cdots a_{n},b_{1}b_{2}\cdots b_{m})$
as the ``$\overline{f}$-distance'' between $a_{1}a_{2}\cdots a_{n}$
and $b_{1}b_{2}\cdots b_{m},$ even though $\overline{f}$ does not
satisfy the triangle inequality unless the strings are all of the
same length. A match $I$ is called a \emph{best possible match} if it realizes the supremum in the definition of $\overline{f}$.

In the case of zero entropy, the $\overline{f}$ distance allows a simple definition of the loosely Bernoulli property.
\begin{definition}[Loosely Bernoulli in the case of zero entropy]
\label{def:zeroLB}
A measure-preserving process $(T,\mathcal{P},\nu)$ is \emph{zero-entropy loosely Bernoulli
}if for every $\varepsilon>0,$ there exists a positive integer $K=K(\varepsilon)$
and a collection $\mathcal{G}$ of ``good'' atoms of $\lor_{1}^{K}T^{-i}\mathcal{P}$
with total measure greater than $1-\varepsilon$ such that for each
pair $A,B$ of atoms in $\mathcal{G},$ $\overline{f}_{K}(x,y)<\varepsilon$
for $x\in A,$ $y\in B.$ 
\end{definition}

\begin{remark}
If this condition is satisfied, then routine estimates show that the
$(T,\mathcal{P},\nu)$ process indeed has zero entropy. Sometimes, zero-entropy loosely Bernoulli transformations are called \emph{standard} or \emph{loosely Kronecker}.
\end{remark}

The following well-known fact from \cite{Fe} and \cite{ORW} as stated in \cite[Property 2.4]{GeKu} will prove useful in our arguments.

\begin{fact}\label{property:omit_symbols}
Suppose $a$ and $b$ are strings of symbols
of length $n$ and $m,$ respectively, from an alphabet $\Sigma$.
If $\tilde{a}$ and $\tilde{b}$ are strings of symbols obtained by
deleting at most $\lfloor\gamma(n+m)\rfloor$ terms from $a$ and
$b$ altogether, 
where $0<\gamma<1$, then 
\begin{equation}
\overline{f}(a,b)\ge\overline{f}(\tilde{a},\tilde{b})-2\gamma.\label{eq:omit_symbols}
\end{equation}
Moreover, if there exists a best possible match between $a$ and $b$ such that no term that is deleted from $a$ and $b$ to form $\tilde{a}$ and $\tilde{b}$ is matched with a non-deleted term, then
\begin{equation}
\overline{f}(a,b)\ge\overline{f}(\tilde{a},\tilde{b})-\gamma.\label{eq:omit_symbols2}
\end{equation}
Likewise, if $\tilde{a}$ and $\tilde{b}$ are obtained by adding at most $\lfloor\gamma(n+m)\rfloor$ symbols to $a$ and $b$, then (\ref{eq:omit_symbols2}) holds. 
\end{fact}

Using the notation from Section \ref{subsec:catfunc} we define the $(n+1)$-blocks $\mathtt{w}^{(n+1)}_j$, $j=1,\dots, s_{n+1}$, in the odometer-based system as follows: 
\[
\mathtt{w}^{(n+1)}_j = \left(\left(\mathtt{w}^{(n)}_1\right)^{(l_n-1)^{2j-1}} \left(\mathtt{w}^{(n)}_2\right)^{(l_n-1)^{2j-1}} \dots \left(\mathtt{w}^{(n)}_{s_n}\right)^{(l_n-1)^{2j-1}}\right)^{(l_n-1)^{2\cdot (s_{n+1}-j)+1}}.
\]
Applying the circular operator $\mathcal{C}_n$ yields the following $(n+1)$-blocks in the corresponding circular system:
\[
\prod^{q_n -1}_{i=0} \left(\prod^{s_n}_{t=1} \left( b^{q_n-j_i} \left(c_n\left(\mathtt{w}^{(n)}_t\right)\right)^{l_n-1}e^{j_i}\right)^{(l_n-1)^{2j-1}}\right)^{(l_n-1)^{2\cdot (s_{n+1}-j)+1}} \ \ \text{ for } \ \ j=1,\dots, s_{n+1}.
\]
Since the newly introduced spacers occupy a proportion of at most $\frac{2}{l_n}$ in substantial substrings of these blocks with at least $\frac{q_{n+1}}{q_n}$ consecutive symbols, we will ignore them in the following consideration. This might decrease the $\overline{f}$ distance between substantial substrings by at most $\frac{4}{l_n}$ according to the aforementioned Fact \ref{property:omit_symbols}. After ignoring the new spacers, the $(n+1)$-blocks in the circular system look as follows for $j=1,\dots, s_{n+1}$:
\[
\left(\left(c_n\left(\mathtt{w}^{(n)}_1\right)\right)^{(l_n-1)^{2j}}\left(c_n\left(\mathtt{w}^{(n)}_2\right)\right)^{(l_n-1)^{2j}} \dots \left(c_n\left(\mathtt{w}^{(n)}_{s_n}\right)\right)^{(l_n-1)^{2j}} \right)^{q_n(l_n-1)^{2(s_{n+1}-j)+1}}.
\]
Apparently, they are of the same form as the patterns in \cite{Fe}. As already observed in \cite{ORW}, the explicit number of blocks and lengths of cycles do not matter for the sake of the argument. It is just important that the number $s_{n+1}$ of types tends to infinity and that the ratio of the length of a cycle in type $j$ and a block of identical types in $j+1$ tends to $0$ rapidly, which in our case corresponds to
\[
\frac{s_n(l_n-1)^{2j}q_n}{(l_n-1)^{2j+2}q_n}=\frac{s_n}{(l_n-1)^2} \to 0 \ \text{ as } \ n \to \infty. 
\]
Hence, we can apply the analysis in \cite[Section 5]{Fe} or \cite[Chapter 10]{ORW} to show that the real-analytic realization $T\in \text{Diff }^\omega_\rho(\T^2,\mu)$ of this strongly uniform circular system is ergodic and not loosely Bernoulli.

To show that there even exists an uncountable family of diffeomorphisms in $\text{Diff }^\omega_\rho(\T^2,\mu)$ which are pairwise not Kakutani-equivalent we follow the approach in \cite[Chapter 12]{ORW}. Let ${\mathbf a}= \left(a_n\right)_{n \in \N}$ be an infinite sequence of zeros and ones. If $a_{n+1}=0$, we construct the $(n+1)$-blocks as described above. In case of $a_{n+1}=1$, we define the $(n+1)$-blocks in the construction sequence of the odometer-based system as
\[
\mathtt{w}^{(n+1)}_j = \left(\left(\mathtt{w}^{(n)}_{s_n}\right)^{(l_n-1)^{2j-1}} \left(\mathtt{w}^{(n)}_{s_n-1}\right)^{(l_n-1)^{2j-1}} \dots \left(\mathtt{w}^{(n)}_{1}\right)^{(l_n-1)^{2j-1}}\right)^{(l_n-1)^{2\cdot (s_{n+1}-j)+1}},
\]
i.\ e.\ we run through the types of $n$-blocks in a cycle in decreasing order this time. As above, we consider the corresponding strongly uniform circular system and realize it as a diffeomorphism $T_{\mathbf a} \in \text{Diff }^\omega_\rho(\T^2,\mu)$. As shown in \cite[Chapter 12, Proposition 3.4]{ORW}, if $a_n \neq b_n$ for infinitely many $n \in \N$, then $T_{\mathbf a}$ and $T_{\mathbf b}$ are not Kakutani-equivalent. Certainly, this gives an uncountable family and proves Theorem \ref{theo:uncountable}. 

\subsection{Real-analytic non-Bernoulli diffeomorphisms with property $K$} \label{subsec:nonBern}
We follow along the lines of \cite{Ka1}, where the first $C^{\infty}$ example of a non-Bernoulli diffeomorphism with property $K$ was constructed. For that purpose, let $A$ be a hyperbolic automorphism of $\T^2$ and let $S\in \text{Diff }^\omega_\rho(\T^2,\mu)$ be an ergodic and not loosely Bernoulli diffeomorphism as constructed in the previous Subsection. We consider its time one suspension $\{S_t\}_{t\in \R}$ on $N= \T^2 \times [0,1] / \sim$, where $(x,1)\sim (S(x),0)$. Then we define the real-analytic diffeomorphism
\[
T: \T^2 \times N \to \T^2 \times N, \ T\left(x,y\right)=\left(Ax, S_{\varphi(x)}(y)\right)
\]
with a real-analytic map $\varphi:\T^2 \to \R$ which is not cohomologous to a constant (i.\ e.\ $\varphi$ cannot be written as $\varphi= \psi \circ A - \psi + c$ with measurable $\psi: \T^2 \to \R$ and $c\in \R$). Then $T$ has property $K$ by \cite[Theorem 1]{Ka1}. To see that $T$ is not Bernoulli we consider the special flow $\{A^{\varphi}_t\}_{t \in \R}$ on $M^{\varphi}=\Meng{(x,s)\in \T^2 \times \R}{0\leq s \leq \varphi(x)}$ and the flow $\tilde{T}_t=A^{\varphi}_t \times S_t$ on $M^{\varphi}\times N$. Since $S$ is not loosely Bernoulli, $\tilde{T}_t$ is not loosely Bernoulli. Moreover, we see that $T$ is the Poincare map of $\tilde{T}_t$ on $M\times\{0\} \times N$. Hence, $T$ as a section of a non-loosely Bernoulli flow cannot be loosely Bernoulli.

\vfill

\begin{center}
    \noindent\textbf{Acknowledgement}
\end{center}

\vspace{.5cm}

\noindent The authors of this paper express their gratitude towards Anatole Katok for asking the question this paper answers. He also helped the authors with numerous mathematical discussions and a steady stream of encouragement. 

Also this paper would not be possible without the discussions the authors had with Matthew Foreman who patiently explained several intricacies of his delicate work and took an interest towards the current result. He also showed the first authors how to produce diagrams using spreadsheet software. The second author also thanks Marlies Gerber for several discussions in the realm of anti-classification results.

\end{document}